\documentclass{amsart}
\usepackage{amsmath}
\usepackage{amsthm}
\usepackage{amssymb}
\usepackage{tikz}
\usetikzlibrary{automata,positioning}
\usepackage{enumitem}
\usepackage{mathrsfs}
\usepackage{url}
\usepackage{hyperref}
\usepackage{mathrsfs}
\usepackage[mathscr]{euscript}
\usepackage{xcolor}
\definecolor{BrickRed}{rgb}{1,0,0}
\definecolor{Black}{cmyk}{0,0,0,1}
\usepackage{soul}
\usepackage{stmaryrd}

\newtheorem{theorem}{Theorem}[section]
\newtheorem{proposition}[theorem]{Proposition}
\newtheorem*{proposition*}{Proposition}
\newtheorem{corollary}[theorem]{Corollary}
\newtheorem*{lemma*}{Lemma}
\newtheorem{alphatheorem}{Theorem}

\newtheorem{alphaconjecture}{Conjecture}

\newtheorem{alphaquestion}{Question}

\theoremstyle{definition}
\newtheorem{definition}[theorem]{Definition}

\newtheorem*{observation*}{Observation}
\newtheorem*{claim*}{Claim}

\newtheorem{example}[theorem]{Example}
\newtheorem{remark}[theorem]{Remark}
\newtheorem*{remark*}{Remark}

\newtheorem*{conjecture*}{Conjecture}
\newtheorem*{convention*}{Convention}
\newtheorem{question}{Question}

\theoremstyle{plain}
\newtheorem{lemma}[theorem]{Lemma}

\newcommand{\bra}[1]{ \left( #1 \right) }
\newcommand{\abs}[1]{\left|#1\right|}
\newcommand{\fp}[1]{\left\{ #1 \right\}}
\newcommand{\br}[1]{\lbr #1 \rbr}

\newcommand{\ifbra}[1]{\left\llbracket #1 \right\rrbracket}

\newcommand{\lbr}{\langle\!\langle}
\newcommand{\rbr}{\rangle\!\rangle}
\newcommand{\fpa}[1]{\left\lVert #1 \right\rVert}

\newcommand{\e}{\varepsilon}
\renewcommand{\a}{\alpha}

\newcommand{\NN}{\mathbb{N}}
\newcommand{\QQ}{\mathbb{Q}}
\newcommand{\Q}{\QQ}

\newcommand{\ZZ}{\mathbb{Z}}
\newcommand{\RR}{\mathbb{R}}
\newcommand{\R}{\RR}
\newcommand{\N}{\NN}
\newcommand{\Z}{\ZZ}
\newcommand{\CC}{\mathbb{C}}
\newcommand{\TT}{\mathbb{T}}
\newcommand{\cA}{\mathcal{A}}
\newcommand{\cB}{\mathcal{B}}

\newcommand{\Nker}{\mathcal{N}}
\newcommand{\Rrev}{\mathrm{R}}
\newcommand{\floor}[1]{\left\lfloor #1 \right\rfloor}

\newcommand{\set}[2]{{\left\{ #1 \ \middle| \ #2 \right\}} }

\newcommand{\FS}{\operatorname{FS}}

\newcommand{\p}{p}

\newcommand{\IP}{$\operatorname{IP}$}
\newcommand{\IPS}{$\mathrm{IPS}$}

\newcommand{\comment}[1]{}

\newcommand{\inter}[1]{\operatorname{int}#1}
\newcommand{\cl}[1]{\operatorname{cl}#1}
\newcommand{\GP}{\mathrm{GP}}

\newcommand{\Haland}{H\r{a}land}

\begin{document}


\author[J. Byszewski \and J. Konieczny]{Jakub Byszewski \and Jakub Konieczny
}

\address[JB]{Department of Mathematics and Computer Science\\Institute of Mathematics\\
Jagiellonian University\\
ul. prof. Stanis\l{}awa \L{}ojasiewicza 6\\
30-348 Krak\'{o}w}
\email{jakub.byszewski@gmail.com}

\address[JK]{Mathematical Institute \\ 
University of Oxford\\
Andrew Wiles Building \\
Radcliffe Observatory Quarter\\
Woodstock Road\\
Oxford\\
OX2 6GG}

\address[JK, current address]{Einstein Institute of Mathematics\\ Edmond J. Safra Campus\\ The Hebrew University of Jerusalem\\ Givat Ram\\ Jerusalem, 9190401\\ Israel}
\email{jakub.konieczny@gmail.com}

\title{Automatic sequences and generalised polynomials}

\maketitle 

\begin{abstract}
	We conjecture that bounded generalised polynomial functions cannot be generated by finite automata, except for the trivial case when they are ultimately periodic.
	
	 Using methods from ergodic theory, we are able to partially resolve this conjecture, proving that any hypothetical counterexample is periodic away from a very sparse and structured set. 
	 In particular, we show that for a polynomial $p(n)$ with at least one irrational coefficient (except for the constant one) and integer $m\geq 2$, the sequence $\floor{p(n)} \bmod{m}$ is never automatic.  
	 
	 We also prove that the conjecture is equivalent to the claim that the set of powers of an integer $k\geq 2$ is not given by a generalised polynomial.
\end{abstract}

\tableofcontents

\section*{Introduction}\label{sec:INT}

\newcommand{\periodic}{b}
Automatic sequences are sequences whose $n$-th term is produced by a finite-state machine from the base-$k$ digits of $n$. (A precise definition is given below.) By definition, automatic sequences can take only finitely many values. Allouche and Shallit \cite{Allouche-Shallit-1992,AlloucheShallit-2003} have generalised the notion of automatic sequences to a wider class of regular sequences and demonstrated its ubiquity and links with multiple branches of mathematics and computer science. The problem of demonstrating that a certain sequence is or is not automatic or regular has been widely studied, particularly for sequences of arithmetic origin (see, e.g., \cite{Allouche-Shallit-1992, AlloucheShallit-2003, Bell2007, ShuYao2011, MedinaRowland2015, Schlage-Puchta2011, Moshe2008, Rowland2010}).

The aim of this article is to continue this study for sequences that arise from generalised polynomials, i.e., expressions involving algebraic operations and the floor function. Our methods rely on a number of dynamical and ergodic tools. A crucial ingredient in our work is one of the main results from the companion paper  \cite{ByszewskiKonieczny2016} concerning the combinatorial structure of the set of times at which an orbit on a nilmanifold hits a semialgebraic subset. This is possible because by the work of Bergelson and Leibman \cite{BergelsonLeibman2007} generalised polynomials are closely related to dynamics on nilmanifolds.

In \cite[Theorem 6.2]{AlloucheShallit-2003} it is proved that the sequence $(f(n))_{n \geq 0}$ given by $f(n)= \lfloor \alpha n + \beta\rfloor$ for real numbers $\alpha, \beta$ is regular if and only if $\alpha$ is rational. The method used there does not immediately generalise to higher degree polynomials in $n$, but the proof implicitly uses rotation on a circle by an angle of $2\pi \alpha$. Replacing the rotation on a circle by a skew product transformation on a torus (as in Furstenberg's proof of Weyl's equidistribution theorem \cite{Furstenberg61}), we easily obtain the following result. (For more on regular sequences, see Section \ref{sec:DEF}.)

\begin{alphatheorem}\label{thm:main-sortof} Let $p\in \R[x]$ be a polynomial. Then the sequence $f(n)=\lfloor p(n) \rfloor, n\geq 0$, is regular if and only if all the coefficients of $p$ except possibly for the constant term are rational.
\end{alphatheorem}

In fact, we show the stronger property that for any integer $m\geq 2$ the sequence $\floor{ f(n) } \bmod m$ is not automatic unless all the coefficients of $p$ except for the constant term are rational, in which case the sequence is periodic. It is natural to inquire whether a similar result can be proven for more complicated expressions involving the floor function such as, e.g., $f(n)= \lfloor \alpha \lfloor \beta n^2 + \gamma\rfloor^2 + \delta n +\varepsilon  \rfloor$. Such sequences are called generalised polynomial  and have been intensely studied (see, e.g., \cite{Haland-1993, Haland-1994, HalandKnuth-1995, BergelsonLeibman2007, Leibman-2012, GreenTaoZiegler-2012, GreenTao-2012}).

Another closely related motivating example comes from the classical Fibonacci word\footnote{We will freely identify words in $\Omega^{\NN_0}$ with functions $\NN_0 \to \Omega$.} $w_{\mathrm{Fib}} \in \{0,1\}^{\NN_0}$, whose systematic study was initiated by Berstel \cite{Berstel-1980,Berstel-1986} (for historical notes, see \cite[Sec.\ 7.12]{AlloucheShallit-book}). There are several ways to define it, each shedding light from a different direction.

\begin{enumerate}
\item\label{item:Fib-A} \textit{Morphic word.} Define the sequence of words $w_0 := 0$, $w_1 := 01$, and $w_{i+2} := w_{i+1} w_{i}$ for $i\geq 0$. Then $w_{\mathrm{Fib}}$ is the (coordinate-wise) limit of $w_i$ as $i \to \infty$.
\item\label{item:Fib-B} \textit{Sturmian word.} Explicitly, $w_{\mathrm{Fib}}(n) = \floor{(2- \varphi)(n+2) }- {\floor{(2- \varphi)(n+1) }}$.
\item\label{item:Fib-C} \textit{Fib-automatic sequence.} If a positive integer $n$ is written in the form $n = \sum_{i =2}^d v_i F_i$, where $v_i \in \{0,1\}$  and there is no $i$ with $v_i = v_{i+1} = 1$, then $w_{\mathrm{Fib}}(n) = v_2$.
\end{enumerate}
The equivalence of \eqref{item:Fib-A} and \eqref{item:Fib-B} is well-known, see, e.g., \cite[Chpt.\ 2]{Lothaire-book}. The representation $v_d v_{d-1}\cdots v_2$ of $n$ as a sum of Fibonacci numbers in \eqref{item:Fib-C} is known as the Zeckendorf representation; it exists for each $n$ and is unique. The notion of automaticity using Zeckendorf representation (or, for that matter, a representation from a much wider class) in place of the usual base-$k$ representation of the input $n$ was introduced and studied by Shallit in \cite{Shallit-1988} (see also \cite{Rigo-2000}), where among other things the equivalence of \eqref{item:Fib-A} and \eqref{item:Fib-C} is shown. We return to this subject in Section \ref{sec:Concluding}. 

Hence, $w_{\mathrm{Fib}}$ gives a non-trivial example of a sequence which is given by a generalised polynomial and satisfies a variant of automaticity related to the Zeckendorf representation. It is natural to ask if similar examples exist for the usual notion of $k$-automaticity. Motivated by Theorem \ref{thm:main-sortof}, we believe the answer is essentially negative, except for trivial examples. We say that a sequence $f$ is \emph{ultimately periodic} if it coincides with a periodic sequence except on a finite set. The following conjecture was the initial motivation for the line of research pursued in this paper. 

\begin{alphaconjecture}
\label{conjecture:main}
	Suppose that a sequence $f$ is simultaneously automatic and generalised polynomial. Then $f$ is ultimately periodic.
\end{alphaconjecture}

In this paper, we prove several slightly weaker variants of Conjecture \ref{conjecture:main}. First of all, we prove that the conjecture holds except on a set of density zero. In fact, in order to obtain such a result, we only need a specific property of automatic sequences. For the purpose of stating the next theorem, let us say that a sequence $f \colon \NN \to X$ is \emph{weakly periodic} if for any restriction $f'$ of $f$ to an arithmetic sequence given by $f'(n) = f(a n + b)$, $a \in \NN,\ b \in \NN_0$, there exist $q \in \NN$, $r, r' \in \NN_0$ with $r \neq r'$, such that $f'(q n + r) = f'(q n + r')$. Of course, any periodic sequence is weakly periodic, but not conversely. All automatic sequences are weakly periodic (this follows from the fact that automatic sequences have finite kernels, see Lemma \ref{lem:auto=>weak-per}). Another non-trivial example of a weakly periodic sequence is the characteristic function of the square-free numbers.

\begin{alphatheorem}\label{thm:main-weakly-periodic}
	Suppose that a sequence $f \colon \NN_0 \to \RR$ is weakly periodic and generalised polynomial. Then there exists a periodic function $b \colon \NN_0 \to \RR$ and a set $Z \subset \NN_0$ of upper Banach density zero such that $f(n) = \periodic(n)$ for $n \in \NN_0 \setminus Z$. 
\end{alphatheorem}

(For the definition of Banach density, see Section \ref{sec:DEF}.)

Theorem \ref{thm:main-weakly-periodic}  is already sufficient to rule out automaticity of many natural examples of generalised polynomials. In particular, sequences such as $\floor{ \sqrt{2} n \floor{ \sqrt{3 n}}} \bmod{10}$ or $\floor{ \sqrt 2 n \floor{ \sqrt{3n}}^2 + \sqrt{5} n + \sqrt{7} } \bmod{10}$ are not automatic. For details and more examples, see Corollary \ref{thm:gen-poly-eqdist=>not-auto}.

To obtain stronger bounds on the size of the ``exceptional'' set $Z$, we restrict ourselves to automatic sequences and exploit some finer properties of generalised polynomials studied in the companion paper \cite{ByszewskiKonieczny2016}.
We use results concerning growth properties of automatic sequences to derive the following dichotomy: If $a \colon \NN_0 \to \{0,1\}$ is an automatic sequence, then the set of integers where $a$ takes the value $1$ is either combinatorially rich (it contains what we call an \IPS\ set) or extremely sparse (in particular, the number of its elements up to $N$ grows as $\log^{r}(N)$ for some integer $r$); see Theorem \ref{thm:Structure-Auto}. This result is especially interesting for sparse automatic sequences, i.e., automatic sequences which take non-zero values on a set of integers of density $0$. 
 Conversely, in \cite{ByszewskiKonieczny2016} we show that sparse generalised polynomials must be free of similar combinatorial structures. As a consequence, we prove the following result.

\begin{alphatheorem}\label{thm:main-optimized}
	Suppose that a sequence $f \colon \NN_0 \to \RR$ is automatic and generalised polynomial. Then there exists a periodic function $b\colon \NN_0 \to \RR$, a set $Z \subset \NN_0$, and a constant $r$ such that $f(n) = \periodic(n)$ for $n \in \NN_0 \setminus Z$ and 
$$ \sup_{M} \abs{ Z \cap [M,M+N)} = O\bra{ \log^{r}(N)}$$	
	as $N \to \infty$ for a certain constant $r$ (dependent on $f$).
\end{alphatheorem}

In fact, we obtain a much more precise structural description of the exceptional set $Z$ (see Theorem \ref{thm:main-B2} for details). Similar techniques allow us to show non-automaticity of some sparse generalised polynomials. For instance, the sequence given by
$$
	n \mapsto 
	\begin{cases}
		1, &\text{if } \fpa{ \sqrt{2} n \floor{ \sqrt{3}n }} < n^{-c};\\
		0, &\text{otherwise} 
	\end{cases}
$$
is not automatic provided that $c$ is small enough. (Here, $\fpa{x}$ denotes the distance of $x$ from $\ZZ$.)
For details, see Example \ref{prop:Heisenberg-exple}.

While Theorem \ref{thm:main-optimized} does not resolve Conjecture \ref{conjecture:main}, our proof thereof greatly restricts the number of possible counterexamples. In fact, in order to prove  Conjecture \ref{conjecture:main}, it would suffice to prove that the characteristic sequence of powers of an integer $k\geq 2$ given by $$g_k(n) =
	\begin{cases}
		1, & \text{ if } n = k^t \text{ for some } t \geq 0;\\
		0, & \text{ otherwise}
	\end{cases}
	$$ is not a generalised polynomial.

\begin{alphatheorem}\label{thm:main-dichotomy}
	Let $k \geq 2$ be an integer. Then exactly one of the following statements holds:
\begin{enumerate}
	\item All sequences that are simultaneously $k$-automatic and generalised polynomial are ultimately periodic.
	\item The characteristic sequence $g_k$ of the powers of $k$ is generalised polynomial.
\end{enumerate}
\end{alphatheorem}

Unfortunately, we are currently unable to decide which of the two possibilities in Theorem \ref{thm:main-dichotomy} holds. Although we expect that $g_k$ should not be a generalised polynomial, in \cite{ByszewskiKonieczny2016} we obtain several examples of algebraic numbers $\lambda > 1$ such that the characteristic function of the set $E_\lambda := \set{ \br{ \lambda^i} }{i \in \NN_0}$ is generalised polynomial, where $\br{x}$ denotes the closest integer to $x$. All our examples are Pisot units (a Pisot number is an algebraic integer $\lambda > 1$ all of whose conjugates have modulus $<1$; a Pisot unit is a Pisot number whose minimal polynomial has constant term $\pm 1$). Conversely, there is no $\lambda > 1$ for which we can prove that the characteristic function of $E_\lambda$ is not given by a generalised polynomial. This prompts us to propose the following question.

\begin{alphaquestion}\label{question:main}
	Suppose that $\lambda > 1$ is such that the characteristic function of the set $E_\lambda := \set{ \br{ \lambda^i} }{i \in \NN_0}$ is given by a generalised polynomial. Is it then necessarily the case that $\lambda$ is a Pisot unit?
\end{alphaquestion}

For a more detailed discussion of this question, see \cite[Section 6]{ByszewskiKonieczny2016}. If $\lambda$ is a Pisot number, then $\br{ \lambda^i}$ obeys a linear recurrence. We show that for such $\lambda$, the characteristic function of $E_\lambda$ cannot be a counterexample to Conjecture \ref{conjecture:main} (see Proposition \ref{prop:recurrent=>not-auto}) except possibly if $\lambda$ is an integer. 

By Theorem \ref{thm:main-dichotomy}, determining the validity of Conjecture \ref{conjecture:main} is equivalent to answering Question \ref{question:main} in the special case when $\lambda$ is an integer.

\subsection*{Contents} 
In Section 1, we discuss some basic notions and results concerning automatic sequences and dynamical systems. We intended this section to be accessible to readers familiar with only one (or neither) of these topics. In Section 2, we prove Theorem \ref{thm:main-sortof} and Theorem \ref{thm:main-weakly-periodic} using methods from topological dynamics. In Section 3, we use known results on growth and structure of automatic sequences to prove that they are either very sparse and structured (in which case we call them arid) or are combinatorially rich. Together with a result about dynamics on nilmanifolds, this allows us to obtain Theorem \ref{thm:main-optimized}. Section 4 contains four seperate topics concerning examples and non-examples of automatic sets and uniform density of symbols in automatic sequences. Section 5 is devoted to the proof of Theorem \ref{thm:main-dichotomy}. Finally, Section 6 discusses some open problems and future research topics.

\subsection*{Acknowledgements} 

The authors thank Ben Green for much useful advice during the work on this project, Vitaly Bergelson and Inger \Haland-Knutson for valuable comments on the distribution of generalised polynomials, and Jean-Paul Allouche and Narad Rampersad for information about related results on automatic sequences.

Thanks also go to Sean Eberhard, Dominik Kwietniak, Freddie Manners, Rudi Mrazovi\'{c}, Przemek Mazur, Sofia Lindqvist, and Aled Walker for many informal discussions.

This research was supported by the National Science Centre, Poland (NCN) under grant no. DEC-2012/07/E/ST1/00185.

Finally, we would like to express our gratitude to the organisers of the  conference \emph{New developments around ${\times 2}\ {\times 3}$ conjecture and other classical problems in Ergodic Theory} in Cieplice, Poland in May 2016 where we began our project. 
\section{Background} \label{sec:DEF}

\subsection*{Notations and generalities}

We denote the sets of positive integers and of nonnegative integers by $\N=\{1,2,\ldots\}$ and $\N_0=\{0,1,\ldots\}$. We denote by $[N]$ the set $[N]=\{0,1,\ldots, N-1\}.$ We use the Iverson convention: whenever $\varphi$ is any sentence, we denote by $\ifbra{\varphi}$ its logical value ($1$ if $\varphi$ is true and $0$ otherwise). We denote the number of elements in a finite set $A$ by $|A|$.

For a real number $r$, we denote its integer part by $\lfloor r \rfloor$,  its fractional part  by $\{r\}=r-\lfloor r \rfloor$,  the nearest integer to $r$ by $\lbr r \rbr = \lfloor r+1/2\rfloor$, and the distance from $r$ to the nearest integer by $\fpa{r} = |r-\lbr r\rbr|$.

We use some standard asymptotic notation. Let $f$ and $g$ be two functions defined for sufficiently large integers. We say that $f=O(g)$ or $f \ll g$ if there exists $c>0$ such that $|f(n)|\leq c \abs{g(n)}$ for sufficiently large $n$. We say that $f=o(g)$ if for every $c >0$  we have  $|f(n)|\leq c |g(n)|$ for sufficiently large $n$. 

For a subset $E\subset \N_0$, we say that $E$ has \emph{natural density} $d(A)$ if $$\lim_{N\to \infty} \frac{|E\cap [N]|}{N} = d(A).$$ 
We say that $E\subset \N_0$ has \emph{upper Banach density} $d^*(A)$ if
 $$\limsup_{N\to \infty} \max_{M} \frac{|E\cap [M,M+N) |}{N} = d^*(A).$$

We now formally define generalised polynomials. 

\begin{definition}[Generalised polynomial]
	The family $\GP$ of \emph{generalised polynomials} is the smallest set of functions $\ZZ \to \RR$ containing the polynomial maps and closed under addition, multiplication, and the operation of taking the integer part. Whenever it is more convenient, we regard generalised polynomials as functions on $\N_0$.
	
	 A set $E\subset \Z$ (or $E \subset \NN_0$) is called \emph{generalised polynomial} if its characteristic function given by $f(n)=\ifbra{n\in E}$ is a generalised polynomial. (Note that this definition depends on whether we are regarding the generalised polynomial as a function on $\Z$ or on $\NN_0$ and a generalised polynomial set $E \subset \NN_0$ might a priori not be generalised polynomial when considered as a subset of $\Z$. It will always be clear from the context which meaning we have in mind.)
\end{definition}

An example of a generalised polynomial is therefore a function $f$ given by the formula $f(n)=\sqrt{3}\lfloor \sqrt{2}n^2+1/7\rfloor^2+n\lfloor n^3+\pi\rfloor$. 

\subsection*{Automatic sequences}

Whenever $A$ is a (finite) set, we denote the free monoid with basis $A$ by $A^*$. It consists of finite words in $A$, including the empty word $\epsilon$, with the operation of concatenation. We denote the concatenation of two words $v,w\in A^*$ by $vw$ and we denote the length of a word $w\in A^*$ by $|w|$. In particular, $|\epsilon|=0$. We say that a word $v\in A^*$ is a factor of a word $w \in A^*$ if there exist words $u,u' \in A^*$ such that $w=uvu'$. We denote by $w^{\Rrev}\in A^*$ the reversal of the word $w\in A^*$ (the word in which the elements of $A$ are written in the opposite order).

Let $k \geq 2$ be an integer and denote by $\Sigma_k=\{0,1,\ldots, k-1\}$ the set of digits in base $k$. For $w\in \Sigma_k^*$, we denote by $[w]_k$ the integer whose expansion in base $k$ is $w$, i.e., if $w=v_l v_{l-1}\cdots v_1 v_0$, $v_i \in\Sigma_k$, then $[w]_k=\sum_{i=0}^l v_ik^i$. Conversely, for an integer $n\geq 0$, we write $(n)_k\in \Sigma_k^*$ for the base-$k$ representation of $n$ (without an initial zero). In particular, $(0)_k = \epsilon$.

The class of automatic sequences consists, informally speaking, of finite-valued sequences $(a_n)_{n\geq 0}$ whose values $a_n$ are obtained via a finite procedure from the digits of base-$k$ expansion of an integer $n$. 

The most famous example of an automatic sequence is arguably the Thue--Morse sequence, first discovered by Prouhet in 1851. Let $s_2(n)$ denote the sum of digits of the base 2 expansion of an integer $n$. Then the Thue--Morse sequence $(t_n)_{n\geq 0}$ is given by $t_n=1$ if $s_2(n)$ is odd and $t_n=0$ if $s_2(n)$ is even.

We will introduce the basic properties of automatic sequences. For more information, we refer the reader to the canonical book of Allouche and Shallit \cite{AlloucheShallit-book}. To formally introduce the notion of automatic sequences, we begin by discussing finite automata.

\begin{definition}\label{automdef1} A deterministic finite $k$-\emph{automaton with output} (which we will just call a $k$-automaton) $\mathcal{A}=(S,\Sigma_k,\delta,s_0,\Omega,\tau)$ consists of the following data: \begin{enumerate}
\item a finite set of states $S$;
\item an initial state $s_{0}\in S$;
\item a transition map $\delta \colon S \times \Sigma_k \to S$;
\item an output set $\Omega$;
\item an output map $\tau \colon S \to \Omega$.
\end{enumerate}
We extend the map $\delta$ to a map $\delta\colon S \times \Sigma_k^* \to S$ (denoted by the same letter) by the recurrence formula $$\delta(s,\epsilon)=s, \quad \delta(s,wv)=\delta(\delta(s,w),v), \quad s\in S, w\in \Sigma_k^*, v\in\Sigma_k.$$ \end{definition}

We call a sequence $k$-\emph{automatic} if it can be produced by a $k$-automaton in the following manner: one starts at the initial state of the automaton, follows the digits of the base-$k$ expansion of an integer $n$, and then uses the output function to print the $n$-th term of the sequence. This is stated more precisely in the following definition.

\begin{definition}\label{automdef2} A sequence $(a_n)_{n\geq 0}$ with values in a finite set $\Omega$ is $k$-\emph{automatic} if there exists a $k$-automaton $\mathcal{A}=(S,\Sigma_k,\delta,s_0,\Omega,\tau)$ such that $a_n=\tau\bra{\delta(s_{0},(n)_k)}$. We call a set $E$ of nonnegative integers \emph{automatic} if the characteristic sequence $(a_n)_{n\geq 0}$ of $E$ given by $a_n=\ifbra{n\in E}$ is automatic.\end{definition}

For some applications, it will be useful to consider the following variant of the definition. A function $\tilde a \colon \Sigma_k^* \to \Omega$ is \emph{automatic} if there exists a $k$-automaton $\mathcal{A}=(S,\Sigma_k,\delta,s_0,\Omega,\tau)$ such that $\tilde a(u) = \tau( \delta(s_0, u ) )$ for $u \in \Sigma_k^*$.

The values of the Thue--Morse sequence are  given by the  $2$-automaton 
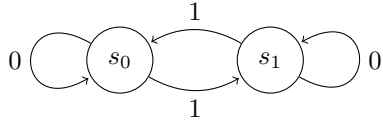
\begin{center}
\begin{tikzpicture}[shorten >=1pt,node distance=2cm, on grid, auto] 
   \node[state] (s_0)   {$s_{0}$}; 
   \node[state] (s_1) [right=of s_0] {$s_1$}; 
  \tikzstyle{loop}=[min distance=6mm,in=210,out=150,looseness=7]
  
    \path[->] 
    
    (s_0) edge [loop left] node {0} (s_0)
          edge [bend right] node [below]  {1} (s_1);
          
 \tikzstyle{loop}=[min distance=6mm,in=30,out=-30,looseness=7]
 \path[->]
    (s_1) edge [bend right] node [above]  {1} (s_0)
          edge [loop right] node  {0} (s_1);
\end{tikzpicture}

\end{center}
with nodes depicting the states of the automaton, edges describing the transition map, $\tau(s_{0})=0$, and $\tau(s_1)=1$. Thus, the Thue--Morse sequence is $2$-automatic.

In the definition above,  the automaton reads the digits starting with the most significant one. In fact, we might equally well demand that the digits be read starting with the least significant digit or that the automaton produce the correct answer even if the input contains some leading zeros. Neither of these modifications changes the notion of automatic sequence \cite[Theorem 5.2.3]{AlloucheShallit-book} (though of course for most sequences we would need to use a different automaton to produce a given automatic sequence).

There is a number of equivalent definitions of the notion of automatic sequence connecting them to different branches of mathematics (stated for example in terms of algebraic power series over finite fields or letter-to-letter  projections of fixed points of uniform morphisms of free monoids). We will need one such definition that has a combinatorial flavour and is expressed in terms of the $k$-kernel.

\begin{definition}\label{automdef3} The $k$-\emph{kernel} $\Nker_k((a_n))$ of a sequence $(a_n)_{n\geq 0}$ is the set of its subsequences of the form $$\Nker_k((a_n))=\{(a_{k^l n+r})_{n\geq 0} \mid l\geq 0, 0\leq r < k^l\}.$$
\end{definition}

Automaticity of a sequence is equivalent to finiteness of its kernel, originally due to Eilenberg \cite{Eilenberg-book-A}. 

\begin{proposition}\cite[Theorem 6.6.2]{AlloucheShallit-book} \label{automthm1} Let $(a_n)_{n\geq 0}$ be a sequence. Then the following conditions are equivalent: \begin{enumerate}
\item The sequence $(a_n)$ is $k$-automatic. 
\item The $k$-kernel $\Nker_k((a_n))$ is finite.\end{enumerate}\end{proposition}

For the Thue--Morse sequence we have the relations $t_{2n}=t_n$, $t_{2n+1}=1-t_n$, and hence one easily sees that the $2$-kernel $\Nker_2((t_n))$ consists of only two sequences $\Nker_2((t_n))=\{t_n, 1-t_n\}$. This gives another argument for the $2$-automaticity of the Thue--Morse sequence.

An automatic sequence by definition takes only finitely many values. In 1992 Allouche and Shalit \cite{Allouche-Shallit-1992} generalised the notion of automatic sequences to the wider class of  $k$-regular sequences  that are allowed to take values in a possibly infinite set. The definition of regular sequences is stated in terms of the $k$-kernel. For simplicity, we state the definition over the ring of integers, though it could also be introduced over  a general (noetherian) ring.

\begin{definition}\label{automdef4} Let $(a_n)_{n\geq 0}$ be a sequence of integers. We say that the sequence $(a_n)$ is $k$-\emph{regular} if its $k$-kernel $\Nker_k((a_n))$ spans a finitely generated abelian subgroup of $\Z^{\N_0}$.
\end{definition}

For example, the following sequences are easily seen to be $2$-regular: $(t_n)_{n\geq 0}$, $(n^3+5)_{n\geq 0}$, $(s_2(n))_{n\geq 0}$. (The corresponding subgroups spanned by the $2$-kernel have rank $2$, $4$, and $2$, respectively. In the case of $t=(t_n)_{n\geq 0}$, the subgroup spanned by the $2$-kernel is free abelian with basis consisting of $t$ and the constant sequence $(1)_{n\geq 0}$.) In fact, every $k$-automatic (integer-valued) sequence is obviously $k$-regular, and the following converse result holds.

\begin{theorem}\label{automthm2}\cite[Theorem 16.1.5]{AlloucheShallit-book}
Let  $(a_n)_{n\geq 0}$ be a sequence of integers. Then the following conditions are equivalent: \begin{enumerate}
\item The sequence $(a_n)$ is $k$-automatic. 
\item The sequence $(a_n)$ is $k$-regular and takes only finitely many values.\end{enumerate}\end{theorem}

\begin{corollary}\label{automthm3}\cite[Corollary 16.1.6]{AlloucheShallit-book}
Let  $(a_n)_{n\geq 0}$ be a sequence of integers that is $k$-regular and let $m\geq 1$ be an integer. Then the sequence $(a_n \bmod m)$ is $k$-automatic.

\end{corollary}

A convenient tool for ruling out that a given sequence is automatic is provided by the pumping lemma.
\begin{lemma}\label{lem:pumping}\cite[Lemma 4.2.1]{AlloucheShallit-book} 	
	 Let $(a_n)_{n\geq 0}$ be a $k$-automatic sequence. Then there exists a constant $N$ such that for any $w \in \Sigma^*_k$ with $\abs{w} \geq N$ and any  integer $0 \leq L \leq \abs{w} - N$ there exist $u_0, u_1, v \in \Sigma^*_k$ such that $v\neq \epsilon$, $w = u_0 v u_1$, $L \leq \abs{u_0} \leq L + N - \abs{v}$, and $a_n$ takes the same value for all $n \in \set{ [u_0 v^t u_1]_k}{t \in \NN_0}$. 
\end{lemma}

The final issue that we need to discuss is the dependence of the notion of $k$-automaticity on the base $k$. While the Thue--Morse sequence is $2$-regular, and is also easily seen to be $4$-regular, it is not $3$-regular. This follows from the celebrated result of Cobham \cite{Cobham-1969}. We say that two integers $k,l\geq 2$ are \emph{multiplicatively independent} if they are not both powers of the same integer (equivalently, $\log k/\log l \notin \Q$).

\begin{theorem}\label{automthm4}\cite[Theorem 11.2.2]{AlloucheShallit-book} Let $(a_n)_{n\geq 0}$ be a sequence with values in a finite set $\Omega$. Assume that the sequence $(a_n)$ is simultaneously $k$-automatic and $l$-automatic with respect to two multiplicatively independent integers $k,l\geq 2$. Then $(a_n)$ is eventually periodic.\end{theorem}

We will have no use for Cobham's theorem. We will, however, use the following much easier related result.

\begin{theorem}\label{automthm5}\cite[Theorem 6.6.4]{AlloucheShallit-book} Let $(a_n)_{n\geq 0}$ be a sequences with values in a finite set $\Omega$. Let $k,l\geq 2$ be two multiplicatively \emph{dependent} integers. Then the sequence $(a_n)$ is $k$-automatic if and only if it is $l$-automatic.\end{theorem}

Let $A$ denote a finite alphabet and let $L$ and $L'$ be languages, i.e., subsets of $A^*$. We denote by $L L' =\{wv \mid w\in L, v\in L'\}$ the concatenation of $L$ and $L'$. For an integer $i\geq 0$, we denote by $L^i = L \cdots L$ the concatenation of $i$ copies of $L$ with the understanding that $L^0 =\{\epsilon\}$. The \emph{Kleene closure} of $L$ is $L^*=\bigcup_{i\geq 0} L^i$. A language $L$  is \emph{regular} if it can be  obtained from the empty set and the letters of the alphabet using the operations of union, concatenation, and the Kleene closure.

Regular languages are intimately connected with automatic sequences via Kleene's theorem \cite{Kleene-1956} (see also \cite[Thm.\ 4.1.5]{AlloucheShallit-book}), which says that a language $L$ over the alphabet $\Sigma_k$  is regular if and only if the sequence  $(a_n)_{n\geq 0}$ given by $a_n=\ifbra{(n)_k \in L}$ is $k$-automatic. 

\subsection*{Dynamical systems}
An (invertible, topological) dynamical system is given by a compact metrisable space $X$ and a continuous homeomorphism $T\colon X\to X$. 
We say that $X$ is minimal if for every point $x\in X$ the orbit $\{T^n x \mid n\in \Z\}$ is dense in $X$. (Equivalently, the only closed subsets $Y\subset X$ such that $T(Y)= Y$ are $Y=X$ or $Y=\emptyset$.) We say that $X$ is totally minimal if the system $(X,T^n)$ is minimal for all $n\geq 1$. 

Let $(X,T)$ be a dynamical system. We say that a Borel measure $\mu$ on $X$ is invariant if for every Borel subset $A\subset X$ we have $\mu(T^{-1}(A))=\mu(A)$. By the Krylov--Bogoliubov theorem (see, e.g., \cite[Thm.\ 4.1]{EinsiedlerWard}), each dynamical system has at least one invariant measure. We say that a dynamical system in \emph{uniquely ergodic} if it has exactly one invariant measure. 

If $(X,T)$ is minimal, $x \in X$, and $U \subset X$ is open, then the set $\set{n \in \ZZ}{T^n x \in U}$ is syndetic, i.e., has bounded gaps \cite[Thm.\ 1.15]{Furstenberg1981}.

We will need the following standard consequence of the ergodic theorem \cite[Thm 4.10]{EinsiedlerWard}, which we also note in \cite[Corollary 1.4]{ByszewskiKonieczny2016}. (Below and elsewhere, $\delta S$ denotes the boundary of the set $S$.) 

\begin{corollary}\label{cor:density-uniform}
	Let $(X,T)$ be a uniquely ergodic dynamical system with the invariant measure $\mu$. Then for any $x \in X$ and any $S \subset X$ with $\mu(\partial S) = 0$, the set $E = \set{n \in \NN_0}{T^n x \in S}$  has upper Banach density $\mu(S)$.
\end{corollary}

In fact, in this case the limit superior in the definition of upper Banach density can be replaced by a limit.

The connection between generalised polynomials and dynamics of  nilsystems has been intensely studied by Bergelson and Leibman in \cite{BergelsonLeibman2007} (see also \cite{Leibman-2012}). 
Nilsystems are a widely studied class of dynamical systems of algebraic origin. Here, we only need several properties which these systems enjoy; in particular, we shall spare the reader the definition of a nilsystem. A good introduction to nilsystems may be found in the initial sections of \cite{BergelsonLeibman2007}. 

A nilsystem $(X,T)$ is minimal if and only if it is uniquely ergodic; the unique invariant measure $\mu_X$ has then full support. If $(X,T)$ is minimal but not totally minimal, then $X$ splits into finitely many connected components $X_1,\dots, X_n$, each $X_i$ is preserved by $T^n$, and each $(X_i,T^n)$ is a totally minimal nilsystem. 

As a special case of the aforementioned connection between nilsystems and generalised polynomials \cite[Thm.\ A]{BergelsonLeibman2007}, we have the following result. (For more details, see also \cite{ByszewskiKonieczny2016}.)

\begin{theorem}[Bergelson--Leibman]\label{BLnilgenpolythm}\label{thm:BergelsonLeibman}
Let $g\colon \Z \to \R$ be a generalised polynomial taking finitely many values $\{c_1,\dots,c_r\}$. Then there exists a minimal nilsystem $(X,T)$  
as well as a point $z \in X$ and a partition $X = S_1 \cup S_2 \cup \ldots \cup S_r$ such that $\mu_X( \partial S_j) = 0$ and 
$$ g(n) = c_j \text{ if and only if } T^n z \in S_j $$
for each $1 \leq j \leq r$. 
\end{theorem}

\begin{remark}\label{rmrk:BLnilgenpolythm}
	Let $g\colon \Z \to \R$ be a generalised polynomial taking finitely many values. Then there exists $a \in \NN$ such that for any $b \in \ZZ$ the generalised polynomial $g_{a,b}(n) := g(an + b)$ has a representation as in Theorem \ref{BLnilgenpolythm} with $(X,T)$ totally minimal.
\end{remark} 
\section{Density 1 results}\label{sec:MAIN} 

\subsection*{Polynomial sequences}
Our first purpose in this section is to prove Theorem \ref{thm:main-sortof}. Recall that we aim to show that the sequence $n \mapsto \floor{ p(n) }$ is not regular if $p(x) \in \RR[x]$ has at least one irrational coefficient other than the constant term. We will show more, namely that the sequence $n \mapsto \floor{p(n)} \bmod{m}$ is not automatic for any $m \geq 2$. 
In fact, we will only need to work with the weaker property of weak periodicity, defined in the introduction. 

\begin{lemma}\label{lem:auto=>weak-per}
	Every automatic sequence is weakly periodic.
\end{lemma}
\begin{proof}
	Let $f$ be a $k$-automatic sequence. Since the restriction of a $k$-automatic sequence to an arithmetic progression  is again $k$-automatic \cite[Theorem 6.8.1]{AlloucheShallit-book}, it will suffice to find $q \in \NN$ and $r,r' \in \NN_0$ with $r \neq r'$ such that $f(qn + r) = f(qn+r')$.
	
	The $k$-kernel $\Nker_k(f)$ of $f$, consisting of the functions $f(k^t n + r)$ for $0 \leq r < k^t$, is finite. Pick $t$ sufficiently large that $k^t > \abs{\Nker_k(f)}$. By the pigeonhole principle, there exist $r \neq r'$ such that $f(k^t n + r) = f(k^t n + r')$. 
\end{proof}
	
The proof of the following proposition is closely analogous to Furstenberg's proof \cite{Furstenberg61} of Weyl's equidistribution theorem \cite{Weyl} (see also \cite[Section 4.4.3]{EinsiedlerWard}).

\begin{proposition}\label{prop:poly=>not-wp}
Let $p(x) \in \RR[x]$ be a polynomial, and let  $m \geq 2$ be an integer. Then the sequence $(\floor{p(n)}\bmod{m})_{n\geq 0}$ is weakly periodic if and only if it is periodic. This happens precisely when all non-constant coefficients of $p(x)$ are rational.
\end{proposition}
\begin{proof}
	If all coefficients of $p(x)$ are rational (except possibly for the constant term) then the sequence  $(\floor{p(n)}\bmod{m})$ is easily seen to be periodic, hence weakly periodic. 

	Now suppose that at least one non-constant coefficient of $p(x)$ is irrational. Replacing $p(x)$ with $p(h x + r)$ for multiplicatively large $h$ and $r=0,1,\ldots,h-1$, we may assume that the leading coefficient of $p(x)$ is irrational. We will prove marginally more than claimed, namely that for any $0 \leq l < m$ the sequence $f$ given by 
	\begin{equation}\label{eq:104}
	f(n) = \ifbra{\floor{p(n)} \equiv l \pmod{m}}
	\end{equation}
	fails to be weakly periodic. For a proof by contradiction, suppose this claim is false for some choice of $l$. 
	
	It will be convenient to expand $p(x)/m = \sum_{i = 0}^d a_i \binom{x}{i}$, where $d = \deg p$, $a_i\in \R$, and $\binom{x}{i} =  x(x-1)(x-2) \cdots (x-i+1)/i!$. Note that $a_d \in \RR \setminus \QQ$ and
	\begin{equation}\label{eq:103}	 f(n)=\ifbra{\frac{p(n)}{m} \bmod{1} \in \left[ \frac{l}{m}, \frac{l+1}{m}\right)}. \end{equation}
	
	We will represent the sequence $p$ dynamically. Let $X$ be the $d$-dimensional torus $\TT^d$ and define the self-map $T \colon X \to X$ by 
	\begin{equation}\label{eq:100}
			(x_1,x_2,x_3,\dots,x_d) \mapsto (x_1 + a_d, x_2 + x_1 + a_{d-1}, 
		 \dots, x_d + x_{d-1} + a_1).
	\end{equation}	
	Put $a_j =0$ for $j>d$. A direct computation shows that for $z = (0,0,\dots,0,a_0)$ and $j=1,\ldots,d$ we have
	\begin{equation}\label{eq:101}
		(T^n z)_j = z_j + \sum_{i \geq 1} a_{d-j+i} \binom{n}{i},
	\end{equation}
and in particular $(T^n z)_d = p(n)/m.$ Putting $A = \TT^{d-1} \times \left[ \frac{l}{m}, \frac{l+1}{m}\right)$, we thus find that 
	\begin{equation}\label{eq:102}
		f(n) = \ifbra{ T^n z \in A}.
	\end{equation}

Since $f$ is weakly periodic, we may find $q$ and $r \neq r'$ such that $f(qn + r) = f(qn + r')$.

	The dynamical system $(X,T)$ can be obtained as a sequence of iterated group extension over an irrational rotation, and hence is totally minimal (this follows easily from the results in, e.g., \cite[Section 4.4.3]{EinsiedlerWard}). In particular, for any point $y \in \cl A$ we may find a sequence $(n_i)_{i \geq 0}$ such that $T^{q n_i + r} z \to y$ and $T^{q n_i + r} z \in A$. It follows that the points $T^{q n_i + r'} z$ converge to $T^{r'-r}y$ and lie in $A$. Thus, $T^{r'-r} (\cl A) \subset \cl A$. In light of total minimality of $T$, this is only possible if $\cl A = X	$ or $\cl A = \emptyset$ --- but this is absurd.
\end{proof}

\begin{corollary}\label{prop:poly=>not-auto}
	With the notation of Proposition \ref{prop:poly=>not-wp}, the sequence $n \mapsto \floor{p(n)} \bmod{m}$ is automatic if and only if it is periodic, and if and only if all the non-constant coefficients of $p(x)$ are rational.
\end{corollary}
\begin{proof}
	Immediate from Proposition \ref{prop:poly=>not-wp} and Lemma \ref{lem:auto=>weak-per}.
\end{proof}

\begin{proof}[Proof of Theorem \ref{thm:main-sortof}]
	Suppose first that all non-constant coefficients of $p(n)$ are rational, and fix an integer $k\geq 2$. Let $h \in \NN$ be such that $h p(n)$ has integer coefficients,  except possibly for the constant term. Then $f_1(n) = \floor{ h p(n) }$ is an integer-valued polynomial, hence is $k$-regular ($\Nker_k(f_1)$ is contained in the $(\deg p + 1)$-dimensional $\ZZ$-module consisting of integer-valued polynomials of degree $\leq \deg p$). Also, $f_2(n) = \floor{ h p(n) } - h f(n) = \floor{h \fp{p(n)}}$ is periodic, hence $k$-automatic, hence $k$-regular. It follows that $f(n) = \frac{1}{h} \bra{ f_1(n) - f_2(n) }$ is regular.

	Conversely, suppose that $f(n)$ is regular. Then by Theorem \ref{automthm2} for any choice of $m \geq 2$ the sequence $f(n) \bmod{m}$ is automatic. Now, it follows from Corollary \ref{prop:poly=>not-auto} that all non-constant coefficients of $p(x)$ are rational. 
\end{proof}

\subsection*{Generalised polynomials}
Having dealt with the case of polynomial maps, we move on to a more general context. Our next goal is to prove Theorem \ref{thm:main-weakly-periodic}. We begin by abstracting and generalising some of the key steps from the proof of Theorem \ref{thm:main-sortof}.

 Recall that a set of integers is \emph{thick} if it contains arbitrarily long segments of consecutive integers, and \emph{syndetic} if it has bounded gaps; every thick set intersects every syndetic set. 

\begin{lemma}\label{lem:tot-min=>not-auto}
	Let $(X,T)$ be a totally minimal dynamical system. Let $A \subset X$ be a set which is neither empty nor dense and such that  $\cl A = \cl \inter A$. Let $z \in X$. Suppose that $f \colon \NN_0 \to \{ 0, 1\}$ is a sequence such that the set of $n$ with $f(n) = \ifbra{T^n z \in A}$ is thick. Then $f$ is not weakly periodic. 
\end{lemma}
\begin{proof}
	Suppose for the sake of contradiction that $f$ is weakly periodic. In particular there exist $q \in \NN$, $r,r' \in \NN_0$ with $r \neq r'$ such that $f(qn + r) = f(qn + r')$. Put $d = r'-r$. 
		
	We will show that  $T^{d} (\cl A) \subset \cl A$. Since $T$ is continuous and $\cl \inter A = \cl A$, it will suffice to prove that $T^{d} (\inter A)  \subset \cl A$. Once this is accomplished, the contradiction follows immediately, because $(X,T^{d})$ is minimal, while $\cl A \neq \emptyset,X$.

	Pick any $y \in \inter A$ and an open neighbourhood $V$ of $T^{d}y$; we aim to show that $V \cap A \neq \emptyset$. Put $U = T^{-d} V \cap \inter A$, and consider the set $S$ of those $n$ for which $T^{q n + r}z \in U$. Since $(X,T^{q})$ is minimal and $U \neq \emptyset$, the set $S$ is syndetic. Let $R_0$ be the set of those $n$ for which $f(n) = \ifbra{T^n z \in A}$ and put $R = \{n\in \NN_0 \mid qn+r \in R_0\}$ and $R' = \{n\in \NN_0 \mid qn+r' \in R_0\}$. 
	
	Since $R_0$ is thick, so is $R \cap R'$. Since $S$ is syndetic, $S \cap R \cap R'$ is non-empty. Pick any $n \in S \cap R \cap R'$ and put $x = T^{qn + r} z$. Since $n \in S$, we have $x \in U \subset A$, and so $T^d x \in V$. Since $n \in R$, we have $f(qn +r) = \ifbra{x \in A} = 1$, and hence also $f(qn + r') = 1$. Finally, since $n \in R'$, we have $1 = f(q n + r') = \ifbra{T^{d} x \in A}$, meaning that $T^{d} x \in V \cap A$. In particular, $ V \cap A \neq \emptyset$, which was our goal.
	\end{proof}

\begin{remark}
	Some mild topological restrictions on the target set $A$ are, of course, necessary in the above lemma. Note that any open, non-dense and non-empty subset of $X$ will satisfy the stated assumptions.
	
	The assumption that the map $T$ is totally minimal is essential. Indeed, take $X$ to be the Thue--Morse shift, i.e., the closed orbit under the shift map of the Thue--Morse sequence. Let \begin{align*} A&=\{(a_n)_{n\in \N_0}\in X \mid a_{2k}=a_{2k+1} \text{ for some }  k\in \N_0\},  \\  B&=\{(a_n)_{n\in \N_0} \in X\mid a_{2k+1}=a_{2k+2} \text{ for some } k\in \N_0\}.\end{align*} Since the Thue--Morse sequence $(t_n)$ has the property $t_{2n} \neq t_{2n+1}$ for all $n$ and since the Thue--Morse word contains no cubes (i.e., no occurences of factors of the form $www$ with $w\in\Sigma_k^*$, $w\neq \epsilon$), we see that $A\cap B \neq \emptyset$, $X=A\cup B$ and $A$ and $B$ are clopen. Let $z=(t_n) \in X$ be the Thue--Morse sequence. Then the function $f(n)=\ifbra{T^n z \in A}$ is periodic with period $2$, and while $X$ is minimal, it is not totally minimal.
\end{remark}

The analogue of the representation of a polynomial sequence using a skew rotation on the torus in \eqref{eq:102} is provided by the Bergelson--Leibman Theorem \ref{thm:BergelsonLeibman}. We are now ready to state and prove the main result of this section, from which Theorem \ref{thm:main-weakly-periodic} easily follows.

\begin{theorem}\label{thm:gen-poly=>not-auto-dens-1}
	Let $g \colon \ZZ \to \RR$ be a generalised polynomial taking finitely many values, and let $f\colon \NN_0 \to \RR$ be a weakly periodic sequence which agrees with $g$ on a thick set $R \subset \NN_0$. Then there exists a set $Z \subset R$ with $d^*(Z) = 0$ such that the common restriction of $f$ and $g$ to $R \setminus Z$ is periodic.	
\end{theorem}

\begin{proof}
	Let the minimal nilsystem $(X,T)$, $z \in X$, and a partition $X = \bigcup_{j=1}^r S_j$ be as in Theorem \ref{thm:BergelsonLeibman}, so that in particular
	\begin{equation}\label{eq:107}
	g(n) = \sum_{j=1}^r \ifbra{T^n z \in S_j} c_j.
	\end{equation}	
	
	If $X$ is not totally minimal, then (as in Remark \ref{rmrk:BLnilgenpolythm}) we may find $a \in \NN$ such that for any $b \in \ZZ$, $g'_{a,b}(n) = g(an + b)$ has a representation as in \eqref{eq:107} on a totally minimal nilsystem. Clearly, $f'_{a,b}(n) = f(an + b)$ is weakly periodic and agrees with $g'_{a,b}(n)$ on the thick set $R'_{a,b} = \{ n \mid an+b \in R\}$. Thus, it will suffice to prove the theorem under the additional assumption that $(X,T)$ is totally minimal.

	We may write 
	\begin{equation}\label{eq:108}
g(n) = \sum_{j=1}^r \ifbra{ T^n z \in \inter S_j}c_j + h(n),
	\end{equation}	
 where $h(n) = 0$ unless $T^n z \in \bigcup_{j=1}^r \partial S_j$. In particular (by Corollary \ref{cor:density-uniform}), the set $Z \subset \NN_0$ of $n$ with $h(n) \neq 0$ has upper Banach density $0$. Note that $R \setminus Z$ is then thick. 
	
	For $j \in \{1,\ldots,r\}$, put $g_j'(n) = \ifbra{T^n z \in \inter S_j}$ and $f'_j(n) = \ifbra{f(n) = c_j}$. Then $g_j'(n) = f_j'(n)$ for $n \in R \setminus Z$. By Lemma \ref{lem:tot-min=>not-auto}, this is only possible if for each $j$, the set $\inter S_j$ is either empty or dense. Since $\mu_X( X \setminus \bigcup_{j=1}^r \inter S_j) = 0$, there is $i$ such that $\inter S_i$ is dense, and $\inter S_j = \emptyset$ for $j \neq i$. Denoting by $Z' \supset Z$ the set of $n \in R$ with $T^n z \in X \setminus \inter S_i$ we have $d^*(Z') = 0$ and $f(n) = g(n) = c_i$ for $n \in R \setminus Z'$, as needed.	
\end{proof}

\begin{proof}[Proof of Theorem \ref{thm:main-weakly-periodic}]
	This is a direct application of Theorem \ref{thm:gen-poly=>not-auto-dens-1} with $f = g$ and $R = \NN_0$
\end{proof}

It is not a trivial matter to determine whether a given generalised polynomial is periodic away from a set of density $0$, although it can be accomplished by the techniques in \cite{BergelsonLeibman2007, Leibman-2012}. In order to give explicit examples, we restrict ourselves to generalised polynomials of a specific form, which is somewhat more general than the one considered in Proposition \ref{prop:poly=>not-wp}. 

\begin{corollary}\label{thm:gen-poly-eqdist=>not-auto}
	Suppose that $q \colon \ZZ \to \RR$ is a generalised polynomial with the property that $\lambda q( an) \bmod{1}$ is equidistributed in $[0,1)$ for any $\lambda \in \QQ \setminus \{0\}$ and $a \in \NN$, and let $m \geq 2$. Then the sequence $f(n) = \floor{ q(n) } \bmod{m}$ is not automatic.	
\end{corollary}
\begin{proof}
	Suppose $f(n)$ were automatic. By Theorem \ref{thm:main-weakly-periodic}, there exist $a \in \NN$ and $Z \subset \NN_0$ with $d^*(Z) = 0$ such that $f(an)$ is constant for $n \in \NN_0 \setminus Z$. Hence, there is some $0 \leq l < m$ such that $\frac{1}{m} q(an) \in \left[ \frac{l}{m}, \frac{l+1}{m} \right)$ for $n \in \NN_0 \setminus Z$, contradicting the equidistribution assumption.
\end{proof}

The uniform distribution of generalised polynomials has been extensively studied by \Haland-Knutson \cite{Haland-1993, Haland-1994, HalandKnuth-1995}, and later a very general theory was developed by Bergelson and Leibman \cite{BergelsonLeibman2007, Leibman-2012}. In view of the the results in \cite{Haland-1993}, it is fair to say that a ``generic'' generalised polynomial $q(n)$ is equidistributed modulo $1$. Hence, the assumptions on $q(n)$ in Corollary \ref{thm:gen-poly-eqdist=>not-auto} are not overly restrictive. 

To make the last remark precise, let us define the (multi)set of coefficients of a generalised polynomial $q$ as follows. If $q(n) = \sum_{j} \alpha_j n^j$ is a polynomial, then the coefficients of $q(n)$ are the non-zero terms among the $\a_j$. If $q(n) = r_1(n) + r_2(n)$ or $q(n) = r_1(n) \cdot r_2(n)$, then the coefficients of $q(n)$ are the union of the coefficients of $r_1(n)$ and $r_2(n)$. Finally, if $q(n) = p(n) \floor{ r(n)}^d$, then the coefficients of $q(n)$ are the union of the coefficients of $r(n)$ and the coefficients of $p(n)$. The set of coefficients will depend on the choice of a representation of the generalised polynomial at hand; we fix one such choice. We cite a slightly simplified version of the main theorem of \cite{Haland-1993}.

\begin{theorem}\label{thm:equidistr-gen-poly}
	Suppose that $q(n)$ is a generalised polynomials, and all of the products of subsets of the coefficients of $q(n)$ are $\QQ$-linearly independent. Then $q(n)$ is equidistributed modulo $1$.
\end{theorem}

As an example of an application, we conclude that $\floor{ \sqrt{2} n \floor{ \sqrt{3} n} } \bmod{10}$ is not an automatic sequence. 
\section{Combinatorial structure of automatic sets}\label{sec:Sparse}
In this section, we begin the investigation of sparse sequences. Here, we call a sequence $f \colon \NN_0 \to \{0,1\} \subset \RR$ \emph{sparse} if it is the characteristic function of a set of density $ 0$ (if such a sequence comes from a generalised polynomial or is automatic, it also has upper Banach density $0$, cf.\ \cite{ByszewskiKonieczny2016} and Lemma \ref{lem:density-equivalence} below). Note that for such sparse sequences, Theorem \ref{thm:main-weakly-periodic} conveys no useful information. Conversely, to prove Conjecture \ref{conjecture:main}, it would suffice (in light of Theorem \ref{thm:main-weakly-periodic}) to verify it for sparse sequences; this observation will be made precise in the proof of Theorem \ref{thm:main-optimized} below.

\subsection*{Arid sets}
To formulate our main result, it is convenient to introduce the following piece of terminology, inspired by Kedlaya \cite{Kedlaya-2006}. Such sets appear in the papers of Szilard--Yu--Zhang--Shallit \cite{SzilardYuZhangShallit-1992}, Gawrychowski--Krieger--Rampersad--Shallit \cite{GawrychowskiKriegerRampersadShallit-2010}, Derksen \cite{Derksen-2007} and Adamczewski--Bell \cite{AdamczewskiBell-2008} (among many others) 
 under different names (regular languages of polynomial growth/sparse/poly-slender/bounded) or without any name. A closely related class of sets known as $p$-normal sets plays a significant r\^ole in the study of zero sets of linear recurrences in positive characteristic; see also \cite{DerksenMasser-2014, AdamczewskiBell-2012}. Other related classes of sets include Saguaro sets of \cite{AdamczewskiBell-2008} and $F$-sets of \cite{MoosaScanlon-2002}. Since we will use the notation simultaneously for languages and for the associated sets of integers, and since some of the existing terminology might be confusing in our context, we have decided to use a different term.
 
\begin{definition}[Arid sets]\label{defverysparse}
	Let $k \geq 2$, $r \geq 0$ be integers. A \emph{basic} $k$-\emph{arid set} (of rank $ \leq r$) is a set of the form
\begin{equation}\label{eq:v-sparse-def-Sigma}
		A = \set{ v_0 w_1^{l_1} v_1 w_2^{l_2} \cdots w_r^{l_r} v_r }{ l_1, \dots, l_r \in \NN_0},
	\end{equation}
	where $v_0,\dots,v_r \in \Sigma_k^*$ and $w_1,\dots,w_r \in \Sigma_k^*$. A set $A \subset \Sigma_k^*$ is $k$-\emph{arid}  (of rank $ \leq r$) if it is a finite union of \emph{basic arid sets} (of rank $ \leq r$). If $k$ is clear from the context, we speak simply of (basic) arid sets.
	
	We similarly define these notions for set of integers: A set $E \subset \NN_0$ is $k$-arid  (of rank $\leq r$)  if it has the form $\set{ [u]_k }{ u \in A}$ where $A \subset \Sigma_k^*$ is arid (of rank $\leq r$).
A sequence $f \colon \NN_0 \to \{0,1\}$ is arid if the set $\set{n \in \NN_0}{f(n) = 1}$ is arid. 
\end{definition}

Using the Kleene star notation, the $k$-arid set $A$ in \eqref{eq:v-sparse-def-Sigma} can be alternatively written as
\begin{equation*}
		A = v_0 w_1^* v_1 w_2^* \cdots w_r^* v_r.
\end{equation*}
In the following, we will not use this notation, and rather use the former notation which seems more appropriate for our context.

\begin{lemma}
	Any $k$-arid  sequence  is $k$-automatic.
\end{lemma}
\begin{proof}
It is clear that any $k$-arid set is given by a regular expression and hence it is $k$-automatic by Kleene's theorem. Alternatively, in this simple case one can construct the required automata by hand.
\end{proof}

Cobham \cite{Cobham-1972} proved that there is a gap in the growth rate of automatic sets.

\begin{proposition} Let $E \subset \N_0$ be a non-empty automatic set. Then exactly one of the following two conditions holds: \begin{enumerate} \item There exists an integer $r\geq 0$ and a real number $c>0$ such that $$\lim_{N\to \infty}  \frac{|E\cap[N]|}{\log^r(N)} = c.$$ \item There exists $\alpha>0$ such that $$\liminf_{N\to \infty} \frac{|E\cap[N]|}{N^{\alpha}}=\infty.$$\end{enumerate}\end{proposition}
\begin{proof} This follows from \cite[Theorem 11 \& 12]{Cobham-1972}\end{proof}

According to the theorem above, automatic sets have either poly-logarithmic or polynomial rate of growth. Szilard--Yu--Zhang--Shallit \cite{SzilardYuZhangShallit-1992} showed that the class of automatic sets of poly-logarithmic growth coincides with the class of arid sets. To state a more precise version of this result, we recall that a state $s$ in a $k$-automaton $\mathcal{A}=(S,\Sigma_k,\delta,s_0,\{0,1\},\tau)$ with output $\{0,1\}$ is called \emph{accessible} if there exists $v\in \Sigma_k^*$ such that $\delta(s_0,v)=s$ and is called $coaccessible$ is there exists $v\in \Sigma_k^*$ such that $\tau(\delta(s,v))=1$.

\begin{proposition}\label{automclassthm}  Let $E \subset \N_0$ be a $k$-automatic set and let $\mathcal{A}=(S,\Sigma_k,\delta,s_0,\{0,1\},\tau)$ be a $k$-automaton with output $\{0,1\}$ that produces $E$, in the sense that an integer $n$ is in $E$ if and only $\tau(\delta(s_0,n))=1$. Then the following conditions are equivalent: 
\begin{enumerate}
\item[(i)] The set $E$ is arid.
\item[(ii)] There exists an integer $r$ such that $|E\cap [N]|= O(\log^r(N))$.
\item[(iii)] There does not exist an accessible and coaccessible state $s\in S$ and $v_1,v_2\in \Sigma_k^*$ such that $v_1v_2\neq v_2v_1$ and $\delta(s,v_1)=\delta(s,v_2)=s$.
\end{enumerate}
Moreover, if $E$ is arid of rank $r$, then the limit $\lim_{N\to \infty} |E\cap[N]|/\log^r(N)$ exists and is finite.
\end{proposition}
\begin{proof} This is essentially proved in \cite{SzilardYuZhangShallit-1992}; our formulation is influenced by \cite[Lemmas 2.1--2.3]{BHS-2017} (for more details and related results see references therein). \end{proof}
\begin{remark*}
	Some similar results are also implicit in \cite[Lemma 6.7]{AdamczewskiBell-2008} and \cite[Proposition 7.9]{Derksen-2007}; see also \cite{Kedlaya-2006}.
\end{remark*}

\begin{remark}\label{kubekkawey} Let $a \geq 1$ be an integer. Then the notions of $k$-arid sets and $k^a$-arid sets coincide. This follows either from a direct argument or from Proposition \ref{automclassthm}.  We will use this observation several times.
\end{remark}

We will in fact need a slight improvement on the information on the rate of growth of arid sets from Proposition \ref{automclassthm}.

\begin{lemma}\label{lem:v-sparse-growth}
	Let $E \subset \NN_0$ be arid of rank (exactly) $r$. Then
	 $$\max_{M \in \NN_0} \abs{E \cap [M,M+N) } = O(\log^r(N)).$$
\end{lemma}
\begin{proof}
	It  suffices to deal with basic arid sets given by
	\begin{equation}\label{eq:110}
		E = \set{ [v_0 w_1^{l_1} v_1 w_2^{l_2} \cdots w_r^{l_r} v_r]_k }{ l_1, \dots, l_r \in \NN_0}.
	\end{equation}

	We begin with some standard reductions. Replacing $w_i$ with suitably chosen powers, altering $v_i$ accordingly, and passing to basic arid subsets, we may assume that all $w_i$ have the same length $a$. Replacing $k$ with $k^a$ and using Remark \ref{kubekkawey} enables us to assume that $\abs{w_i} = 1$ for each $i$. If $r$ is minimal, we further know that if $w_i = w_{i+1}$ for some $i$, then $v_i$ is not a power of $w_i$. Finally, we may assume  that $N = k^{L}$ is a large power of $k$, and that $M = k^L M'$ is divisible by $N$.
	
	Since an element of $E \cap [M,M+N)$ is uniquely determined by its final $L$ digits, the bound $\abs{E \cap [N]} \ll L^r$ follows immediately from counting the $r$-tuples $(l_1, \dots, l_r)$ with $\sum_{i=1}^r l_i + \sum_{i=0}^r \abs{v_i} \leq L$. 
\end{proof}

We are now ready to state the main theorem of this section in a more convenient language.
\begin{theorem}\label{thm:main-B2}
Suppose that a sparse set $E \subset \NN_0$ is simultaneously $k$-automatic and generalised polynomial. Then $E$ is $k$-arid. 
 
\end{theorem}

For the proof of this result, we need to use the notion of IPS sets introduced in \cite{ByszewskiKonieczny2016}.

\subsection*{IPS sets and automatic sequences}
The following notion generalises the classical notion of an \IP\ set that is of importance in combinatorial number theory and ergodic theory (for origin of the term \IP, which stands either for infinite-dimensional parallelepiped or idempotent, see, e.g., \cite{BergelsonLeibman-IP}). This notion is discussed in more detail in \cite{ByszewskiKonieczny2016} (in particular, an equivalent definition of \IPS\ sets in terms of ultrafilters is given there).

\begin{definition}[\IP\ and \IPS\ sets]\label{def:IPS-set}
	For a sequence $(n_i)_{i \in \NN} \subset \NN$, the corresponding set of \emph{finite sums} is 
\begin{equation}\label{eq:IP-def}
\FS(n_i) = \set{ n_\a }{ \a \subset \NN,\ 0 < \abs{\a} < \infty},
\end{equation}  
	where $n_\a = \sum_{i \in \a}n_i$. Any set containing a set of the form $\FS(n_i;N_t)$ for some $(n_i),\ (N_t)$ is called an \emph{\IPS\ set}.

	For a sequence $(n_i)_{i \in \NN} \subset \NN$ and shifts $(N_t)_{t \geq 1} \subset \NN_0$, the corresponding set of \emph{shifted finite sums} is 
\begin{equation}\label{eq:IPS-def}
\FS(n_i;N_t) = \set{ n_\a + N_t }{ t \in \NN,\  \a \subset \{1,2,\ldots, t\}, \a \neq \emptyset},
\end{equation}  
	where again $n_\a = \sum_{i \in \a}n_i$. Any set containing a set of the form $\FS(n_i;N_t)$ for some $(n_i),\ (N_t)$ is called an \emph{\IPS\ set}.
\end{definition}

\begin{example}\label{ex:IPS}
	Fix $k \geq 2$. Let $v_1, v_2 \in \Sigma_k^*$ be two distinct words with $\abs{v_1} = \abs{v_2} = l$, and let $u_0,u_1 \in \Sigma_k^*$ be arbitrary. Consider the set 
	$$
		E = \set{ [u_0 v_{j_1} v_{j_2} \cdots v_{j_t} u_1] }{ {j_i} \in \{1,2\} \text{ for } 1 \leq i \leq t \text{ and } t \geq 0}.
	$$
	Then $E$ is an \IPS\ set. Indeed, $E = \FS( n_i; N_t)$, where $N_t = [u_0 v_{1}^t u_1]_k$ and $n_i$ = $([v_2]_k - [v_1]_k)k^{(i-1)l + \abs{u_1}}$ (assuming, as we may, that $[v_2]_k > [v_1]_k$).
	If $[u_0]_k = [u_1]_k = [v_1]_k = 0$, then $E$ is an \IP\ set. 
\end{example}

\IPS\ sets occur in our work due to the following result. 

\begin{theorem}\label{thm:Structure-Auto}
	Let $E \subset \NN_0$ be an automatic set. Then either $E$ is arid or it is \IPS.
\end{theorem}
\begin{proof} Assume that $E \subset \NN_0$ is automatic but not arid; we need to show that $E$ is \IPS. Let $\mathcal{A}=(S,\Sigma_k,\delta,s_0,\{0,1\},\tau)$ be a $k$-automaton with output $\{0,1\}$ which produces the characteristic sequence of $E$ when reading digits starting from the most significant one and ignoring the initial zeros.
	
	Since $E$ is not arid, neither is the set $A = \set{ w \in \Sigma^*_k}{ \tau( \delta(s_0,w)) = 1}$. Hence, by Proposition \ref{automclassthm}, there exists an accessible and coaccessible state $ s \in S$ and $v_1,v_2\in \Sigma_k^*$ such that $v_1v_2\neq v_2v_1$ and $\delta(s,v_1)=\delta(s,v_2)=s$. Replacing $v_1$ and $v_2$ by their powers and interchanging them if necessary, we may assume that $v_1$ and $v_2$ are of equal length $l = \abs{v_1} = \abs{v_2}$ and $[v_1]_k < [v_2]_k$. Pick $u_0, u_1 \in \Sigma_k^*$ so that $s= \delta(s_0,u_0)$, and $\tau(\delta(s,u_1))=1$. 
	
	The set $A$ contains all words of the form $w = u_0 v_{j_1} v_{j_2} \dots v_{j_t} u_1$, where ${j_i} \in \{1,2\}$ and $t \in \NN_0$. It follows that $A$ is \IPS\ (cf.\ Example \ref{ex:IPS}).
\end{proof}

In order to prove Theorem  \ref{thm:main-B2}, we need to recall one of the main results of \cite{ByszewskiKonieczny2016} (Theorem A), whose proof uses ergodic theory and the machinery of ultrafilters.

\begin{theorem}\label{thm:GP-vs-IPS}
	Let $E \subset \ZZ$ be a sparse generalised polynomial set. Then $E$ is not \IPS.
\end{theorem}

Theorem  \ref{thm:main-B2} and Theorem   \ref{thm:main-optimized} now follow quite easily.

\begin{proof}[Proof of Theorem \ref{thm:main-B2}]
Let $E$ be the set in Theorem \ref{thm:main-B2}. By Theorem \ref{thm:GP-vs-IPS}, $E$ is not \IPS. Hence, by Theorem \ref{thm:Structure-Auto}, it is arid.
\end{proof}

\begin{proof}[Proof of Theorem \ref{thm:main-optimized}]
	Suppose that $f \colon \NN_0 \to \RR$ is automatic and generalised polynomial. Let $\periodic(n)$ be the periodic function such that the set $Z = \set{n \in \NN_0}{f(n) \neq \periodic(n)}$ has $d^*(Z) = 0$. (The existence of $\periodic(n)$ is guaranteed by Theorem \ref{thm:main-weakly-periodic}.)
	
	Note that $Z$ is generalised polynomial and automatic (automaticity is clear; to see that $Z$ is generalised polynomial, compose $f-b$ with a polynomial $p$ such that $p(0) = 1$ and $p(x-y) = 0$ for $x \in f(\NN_0)$, $y\in \periodic(\NN_0)$, $x \neq y$).
	
	By Theorem \ref{thm:main-B2}, $Z$ is arid. Hence, by Lemma \ref{lem:v-sparse-growth} below, we have 
	$$\abs{ Z \cap [M,M+N) }  = O\bra{\log^r(N)}$$
	 for some $r \in \NN_0$ as $N \to \infty$.
\end{proof}

If Conjecture \ref{conjecture:main} is true then there are no nontrivial examples of arid generalised polynomial sets (indeed, by Theorem \ref{thm:main-dichotomy} non-existence of such sets is precisely equivalent to Conjecture \ref{conjecture:main}; see also Proposition \ref{automverysparse}). However, there are examples of generalised polynomial sets which exhibit some properties reminiscent of arid sets. We have already mentioned in this context that the set of Fibonacci numbers is a generalised polynomial set, and in  \cite[Theorems B \& C]{ByszewskiKonieczny2016} we have extended this to certain linear recurrences of order $2$ and $3$ as well as arbitrary sets whose size grows at a sublogarithmic rate.

It is important to note that in the statement of Theorem \ref{thm:Structure-Auto} it is not possible to replace \IPS\ sets with \IP\ sets or their translates (cf.\ Example \ref{ex:notIPrich}). We discuss this question further in the next section.

\section{Examples and properties of automatic sets}\label{sec:S-AUTO}

\subsection*{$\cB$-free sets}

In this subsection, we will discuss a simple class of examples of automatic sets, the $\cB$-free sets, which will allow us to show that in the statement of Theorem \ref{thm:Structure-Auto} it is in general not possible to replace \IPS\ sets with translates of \IP\ sets (Example \ref{ex:notIPrich}). 

\begin{example}\label{ex:B-free}
	Let $k \geq 2$, and let $\cB \subset \Sigma_k^*$ be a finite set of  `prohibited' words of length $\leq t$. A word $u \in \Sigma_k^*$ is $\cB$-free if $u$ contains no $b \in \cB$ as a factor. Accordingly, $n \in \NN_0$ is $\cB$-free if its base-$k$ expansion $(n)_k$  is $\cB$-free. Denote the set of $\cB$-free integers by $F_\cB$.
	
\begin{enumerate}
\item The set $F_\cB$ is $k$-automatic.
\item If $\cB \neq \emptyset$, then $F_\cB$ is sparse.
\item\label{wsciekly} If $ \sum_{b \in \cB} k^{-\abs{b}} \leq \frac{1}{16t}$, then $F_\cB$ is not arid.
\item If each $b \in \cB$ contains at least two non-zero digits, then $F_\cB$ is \IP.
\item\label{IPrichonlyzeroes} If some $b \in \cB$ consists only of $0$'s, then $F_\cB - m$ is not \IP\ for any $m \in \ZZ$.
\end{enumerate}
\end{example}

\begin{proof}
\begin{enumerate}
\item It is not difficult to explicitly describe a $k$-automaton which computes the characteristic function of $F_\cB$; alternatively, the claim follows immediately from Kleene's theorem.
\item We may assume that $\cB$ consists of a single string of length $t$. Then the probability that a randomly chosen word of length $m$ does not contain $b$ is at most $\bra{1 - {k^{-t}}}^{\floor{ m/t}}$. The claim easily follows from this.

\item We may assume $\cB \neq \emptyset$. Construct an undirected graph $G = (V,E)$ (we allow $G$ to have loops), where $V = \Sigma_k^{t}$, and $\{u,v\} \in E$ if $uv$ and $vu$ are both $\cB$-free. If $u_1, u_2, \dots, u_r$ is a walk in $G$, then $u_1 u_2 \cdots u_r$ is $\cB$-free. Assume that $G$ contains a walk $u_1,w,u_2$ of length $2$ with $u_1 \neq u_2$. With loss of generality, we may assume that $u_1\neq 0^t$ (otherwise, switch $u_1$ and $u_2$). Then for any $i_1,\dots,i_r \in \{1,2\}$ the word $v = u_1 w u_{i_1} w u_{i_2} w \cdots u_{i_r} w$ is $\cB$-free. Hence, $[v]_k \in F_\cB$ and we can see either directly or from Proposition \ref{automclassthm}  that $F_\cB$ is not arid. Thus, it remains to check that $G$ contains a length $2$ walk with distinct endpoints; for the sake of contradiction suppose that this is not the case.

Since each vertex has at most one neighbour  (including itself if $\{u,u\}$ is an edge), the graph is a disjoint union of paths of length $1$, loops, and vertices, and hence $\abs{E} \leq \abs{V} = k^{t}$. On the other hand, given $b \in \cB$, the number of pairs $(u,v) \in V^2$ such that $b$ appears in $uv$ or $vu$ is $< 4t k^{2t -\abs{b}}$, so 
$$
	\abs{E} \geq \frac{k^t(k^t+1)}{2} - 4 t k^{2t} \sum_{b \in \cB} k^{-\abs{b}}
	>   \frac{k^{2t}}{4}  \geq k^{t},
$$
(note that the assumption implies that $k^t \geq 16$), which gives a contradiction.

\item Let $n_i = k^{it}$. Then $\FS(n_i) \subset F_\cB$.
\item Suppose that $F_\cB$ contains $E+m$ for some \IP\ set $E$ and integer $m$. Replacing $E$ with a smaller \IP\ set if necessary, we may assume that $m > 0$. Since $E$ is \IP, for any $l \geq 0$ there exists $n \in E$ which is divisible by $k^l$. If $l$ is large enough (it suffices that $l>t+\floor{\log N/\log k}$) then $n+m$ is an element of $F_\cB$ whose base-$k$ expansion contains $t$ consecutive zeros, contradicting the assumption on $\cB$.\qedhere\end{enumerate}\end{proof}	
\begin{remark}
	A similar example was considered by Miller \cite{Miller-2012}, who gave sufficient conditions for $F_\cB$ to be infinite.
\end{remark}

\begin{example}\label{ex:notIPrich}
The set
	$$
		F_{00} = \set{n \in \NN_0}{\text{the binary expansion of $n$ does not contain $00$}}	$$
is $2$-automatic, sparse, not arid, and does not contain a translate of an \IP\ set.
\end{example}
\begin{proof} We see that $F_{00}$ is not arid by  Proposition \ref{automclassthm} or by a simple modification of the proof of \ref{ex:B-free}.\eqref{wsciekly}. The remaining claims follow directly from Example \ref{ex:B-free}.\end{proof}

The following two examples can be verified similarly.
	
\begin{example}\label{ex:Fib-IPrich}
The set
	$$
		F_{11}=  \set{n \in \NN_0}{\text{the binary expansion of $n$ does not contain $11$}}	$$
is $2$-automatic, sparse, not arid, and \IP.
\end{example}

\begin{example}\label{ex:BS-sparse-IPrich}
	The Baum--Sweet sequence (\cite{BaumSweet}) given by
	$$
		f_{\mathrm{BS}}(n) ={\ifbra{\text{the binary expansion of $n$ does not contain $10^l1$ for an odd integer $l$}}}.$$
		It takes the value $1$ on a set which is $2$-automatic, sparse, not arid, and \IP.
\end{example}

\subsection*{Translates of \IP\ sets}

Even though in general non-arid automatic sets need not contain translates of \IP\ sets, this is nevertheless the case under certain stronger assumptions on the set. 

\begin{proposition}\label{automshiftedipset} 
	Let $E \subset \NN_0$ be a $k$-automatic set. Assume that for every $w\in \Sigma_k^*$ there is an integer $n\in E$ such that $w$ is a factor of $(n)_k$. Then the set $E-m = \set{n-m}{n \in E}$ is \IP\ for some $m \in \N_0$.
\end{proposition}

\begin{proof}[Proof of Proposition \ref{automshiftedipset}]

Let $\mathcal{A}=(S,\Sigma_k,\delta,s_0,\{0,1\},\tau)$ be a $k$-automaton that produces the characteristic sequence of $E$ by reading the digits of $n$ \emph{starting with the least significant one}, allowing for leading zeros. We will denote the word $0\cdots 0\in \Sigma_k^*$ with $n$ zeros by $0^n$. We begin by proving the following claim.

\begin{claim*}
	There exist states $s,s'\in S$ with $\tau(s)=1$, an integer $l\in \N$, and a word $v \in \Sigma_k^l$ that is not a power of $0$ such that for $z = 0^l$ we have $\delta(s,z)=s'$, $\delta(s,v)=s$, $\delta(s',z)=s'$, $\delta(s',v)=s$. This is portrayed below: 
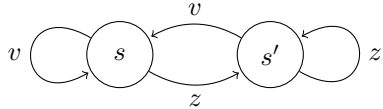
\begin{center}\begin{tikzpicture}[shorten >=1pt,node distance=2cm, on grid, auto] 
   \node[state] (s)   {$s$}; 
   \node[state] (s') [right=of s] {$s'$}; 
  \tikzstyle{loop}=[min distance=6mm,in=210,out=150,looseness=7]
  
    \path[->] 
    
    (s) edge [loop left] node {$v$} (s)
          edge [bend right] node [below]  {$z$} (s');
          
 \tikzstyle{loop}=[min distance=6mm,in=30,out=-30,looseness=7]
 \path[->]
    (s') edge [bend right] node [above]  {$v$} (s)
          edge [loop right] node  {$z$} (s');
\end{tikzpicture}\end{center} 
\end{claim*}

\begin{proof}[Proof of the claim]

Let $n=|S|$ be the number of states in $\mathcal{A}$. 
We first show a weaker statement, namely that there is a state $s$ with $\tau(s) = 1$ such that if $\tilde{s} =  \delta(s,0^n)$ denotes the state reached from $s$ after reading $n$ zeros, then we can return from $\tilde{s}$ to $s$ along a path not consisting only of zeros, that is $ \delta(\tilde{s} ,\tilde{v}) = s$ for some $\tilde{v} \in \Sigma_k^*$ that is not a power of $0$.

To prove this, we construct a word $w = w_1 w_2 \cdots w_{n^2}$ as follows. Enumerate all pairs in $S \times S$ as $(s_i, s_i')$ for $1 \leq i \leq n^2$. In the first step, if $\tilde{s}_1$ is reachable from $s_1$, let $w_1$ describe any path between the two, so that $ \delta(s_1, w_1) = s_1'$; otherwise, let $w_1 = \epsilon$. In general, if $w_1, \dots, w_{i-1}$ have been defined, choose $w_i$ so that $ \delta( s_i, w_1w_2 \cdots w_{i-1}w_i) = s_i'$ if possible (i.e., if $s_i'$ is reachable from $ \delta( s_i, w_1w_2 \cdots w_{i-1})$), and $w_i = \epsilon$ otherwise.

By the  assumption on the set $E$, there exists some $x, y \in \Sigma_k^*$ such that for $s =  \delta( s_0, xwy)$ we have $\tau(s) = 1$. Applying the same assumption with $w1$ in place of $w$, we may ensure that $y$ is not a power of $0$. It remains to show that we can return from $\tilde{s} =  \delta(s, 0^n)$  to $s$. For $0 \leq i \leq n^2$, let $r_i = \delta( s_0, x w_1 w_2 \cdots w_i)$ denote the intermediate states on the path from $s_0$ to $s$ labelled $xwy$, in particular $r_0 = \delta( s_0, x)$. 
The construction of $w$ is arranged so that for any $i$ with $s_i = r_0$, we have $r_{i} =  \delta( r_{i-1}, w_i) = \tilde{s}_i$, provided that $s_i'$ is reachable from $r_{i-1}$.

Choose $1\leq j\leq n^2$ such that $s_j = r_0$ and $s_j' = \tilde{s}$. Since $s$ is reachable from $r_{j-1}$ and $\tilde{s}$ is reachable from $s$, $\tilde{s}$ is reachable from $r_{j-1}$. Hence, the construction of $w$ guarantees that $r_{j} = \delta( r_{j-1}, w_{j}) = \tilde{s}$. In particular, $ \delta(\tilde{s}, \tilde{v}) = s$, where $\tilde{v} = w_{j+1} \dots w_{n^2}y$. Note that $\tilde{v}$ is not a power of $0$ since neither is $y$. This proves the weaker version of the claim.

To prove the stronger statement, note first that since $S$ has only $n$ states, there exist $0\leq i<j\leq n$ such that $ \delta(s,0^i)=\delta(s,0^j)$. Let $ m> i$ be any integer divisible by $(j-i)$ and put $s'=\delta(s,0^m)$. Since $m$ is divisible by $(j-i)$, we have $s'=\delta(s',0^{m})$. Because $\tilde{s}$ is reachable from $s'$ (actually, $ \delta(s', 0^n) = \tilde{s}$), there is a word $u$ (equal to $ 0^n v$, hence not a power of $0$) such that $ \delta (s', u) = s$. Take $v =  (0^m u)^{m}$ and $l = m(\abs{u} + m)$. The states $s,s'$ and the word $v$ (of length $l$) satisfy all the required conditions, namely $\delta(s,0^l)=s'$, $\delta(s, v)=s$, $\delta(s',0^l)=s'$, $\delta(s',v)=s$, and $\tau(s)=1$.
\end{proof}

To finish the proof of Proposition \ref{automshiftedipset}, we may assume that all states in $\mathcal{A}$ are accessible. Choose states $s$ and $s'$ and words $v$ and $z=0^l$ as in the statement of the claim. Let $u \in \Sigma_k^*$ be such that $\delta(s_0, u) = s$. For any word $w = u v_1 v_2 \cdots v_r$, where $v_i \in \{v,z\}$ for $1 \leq i < r$ and $v_r = v$, we have $ \delta (s_0, w) = s$, whence $[w^{\Rrev}]_k \in E$. It follows that $E$ contains $\FS(n_i; N)$, where $N = [u^{\Rrev}]_k$ and $n_i = k^{(i-1) l + \abs{u}} [v^{\Rrev}]_k$, $i\in\N$.
\end{proof}

Proposition  \ref{automshiftedipset} has the following amusing application which, however, does not require the full strength of Theorem \ref{thm:GP-vs-IPS}. (Similar results can be shown in greater generality.)

\begin{example}\label{prop:Heisenberg-exple}
	There exists a constant $c>0$ such that for any sequence $\varepsilon(n)$ which  is a rational power of a generalised polynomial such that $\varepsilon(n)\ll n^{-c}$ as $n \to \infty$, the set
	$$E = \set{n \in \NN_0}{ \fpa{ \sqrt{2} n \floor{ \sqrt{3 n}} } < \e(n) }$$
is not automatic.
\end{example}
\begin{proof}
	It is shown in \cite[Propositions 4.6 \& 4.8]{ByszewskiKonieczny2016} that $E$ is generalised polynomial, $E$ contains no translate of an \IP\ set, and that $E \cap (a \NN + b) \neq \emptyset$ for any $a\in \NN$, $b \in \NN_0$. 
	
	Suppose that $E$ were $k$-automatic. Since $E$ intersects nontrivially any arithmetic progression, it would satisfy the assumptions of Proposition \ref{automshiftedipset}, and thus would contain a translate of an \IP\ set, contradicting the previously mentioned results.	 
\end{proof}

\subsection*{Densities of symbols}  In this subsection, we prove a lemma on densities of occurrences of symbols in automatic sequences. As a corollary, we obtain the claim that  sparse automatic sequences take non-zero value at a set of \emph{Banach} density $0$.

The density of symbols for an automatic sequence is often uniform. A set $E \subset \NN_0$ has \emph{uniform density} $d = d(E)$ if $\abs{E \cap [M,M+N)}/ N \to d$ as $N \to \infty$ uniformly in $M$. For an automaton $\cA = (S,\Sigma_k,\delta,s_0,\Omega,\tau)$, a \emph{strongly connected component} is an automaton $\cA' =  (S',\Sigma_k,\delta',s'_0,\Omega,\tau')$, where $S' \subset S$ is non-empty, preserved under $\delta(\cdot,j)$ for all $j \in \Sigma_k$ and minimal with respect to these properties, $s'_0 \in S'$, and $\delta',\ \tau'$ are the restrictions of $\delta$ and $\tau$ to $S'$, respectively.

\begin{lemma}\label{lem:density-equivalence}
	Let $a\colon \NN_0 \to \Omega$ be a $k$-automatic sequence generated by an automaton $\cA =  (S,\Sigma_k,\delta,s_0,\Omega,\tau)$ reading input starting with the most significant digit, ignoring the initial zeros, and such that all the states are accessible. For $y \in \Omega$, let $\rho_y \geq 0$. Then the following conditions are equivalent:
	\begin{enumerate}
	\item\label{cond:54-1} For any $y \in \Omega$, the set $\set{n \in \NN_0	}{a(n) = y}$ has density $\rho_y$;
	\item\label{cond:54-2} For any $y \in \Omega$, the set $\set{n \in \NN_0	}{a(n) = y}$ has uniform density $\rho_y$;
	\item\label{cond:54-3} For any sequence $\tilde a' \colon \Sigma_k^* \to \Omega$ produced by a strongly connected component $\cA'$ of $\cA$ and for any $y \in \Omega$ we have
	$$
	 \abs{\set{ u \in \Sigma_k^L }{ \tilde a'(u) = y }}/k^L \to \rho_y \text{ as } L \to \infty.
	$$ 
	\end{enumerate}
\end{lemma}
\begin{proof}
	It is clear that \eqref{cond:54-2} implies \eqref{cond:54-1}. We will  show that \eqref{cond:54-1} implies \eqref{cond:54-3} and \eqref{cond:54-3} implies \eqref{cond:54-2}. Throughout, it will be convenient to assume that $\Omega = \{0,1\}$, which we may do without loss of generality. We then write $\rho$ for $\rho_1$.
	
	Suppose that \eqref{cond:54-1} holds, and take some $\tilde a'$ as in \eqref{cond:54-3}. There is some $v \in \Sigma_k^*$ such that $\tilde a'(u) = a([vu]_k)$, whence	 
\begin{align}
	\frac{1}{k^L} \sum_{u \in \Sigma_k^L} \tilde a'(u)
	=  \frac{1}{k^L} \sum_{n = [v]_k k^L}^{([v]_k+1) k^L-1} a(n) \to  \rho 
	\label{eq:762}
\end{align}
as $L \to \infty$. 

	Now suppose that \eqref{cond:54-3} holds.
	For any $N,M$ and $L$, we have
\begin{align}
	\frac{1}{N} \sum_{n = M}^{M+N-1} a(n) &= 
	 \frac{1}{N} \sum_{m=\floor{M/k^L} }^{\floor{(M+N)/k^L}} \sum_{n = 0 }^{k^L-1} a(m k^L + n) + O(k^L/N) 	\label{eq:764}
\end{align} uniformly in $M$. For any $m \in \NN_0$, consider the sequence $\tilde a'_m \colon \Sigma_k^* \to \Omega$ given by $\tilde a'_m(u) = a( [(m)_ku]_k)$, so that if $u \in \Sigma_k^L$, then $a(m k^L + [u]_k) = \tilde a'_m(u)$. Note that $\tilde a'_m$ is produced by the automaton $\cA'$ that is obtained from $\cA$ by changing the initial state to $s'_{0} = \delta(s_{0},(m)_k)$.
	
	If $\delta(s_{0},(m)_k)$ lies in a strongly connected component of $\cA$, then we may use \eqref{cond:54-3} to estimate the inner sums in \eqref{eq:764}:
	$$\sum_{n = 0 }^{k^L-1} a(m k^L + n) = k^L \rho + o(k^L)$$
 as $L \to \infty$ (where the error term is uniform with respect to $m$, since there are only finitely many possible sequences $\tilde a'_m$). It is an easy exercise to check that the set of $m \in \NN_0$ such that $\delta(s_{0},(m)_k)$ does not lie in a strongly connected component of $\cA$ has upper Banach density $0$. Estimating the inner sums in \eqref{eq:764} corresponding to such $m$ trivially by $O(k^L)$, and letting $L \to \infty$ slowly enough  so that $k^L/N \to 0$, we conclude that 
 $$\frac{1}{N} \sum_{n = M}^{M+N-1} a(n) = \rho + o(1)$$
 as $N \to \infty$ uniformly in $M$. Hence, \eqref{cond:54-2} holds. 		
\end{proof}

\subsection*{Linear recurrence sequences}\label{sec:Recurrent}

We have already noted that the set of values of a linear recurrence sequence can be a generalised polynomial set. This is the case for the Fibonacci sequence; for more information, see \cite[Theorem B]{ByszewskiKonieczny2016}. In contrast, we show that the set of values of a linear recurrence sequence is not automatic, except for trivial examples.
In the proof, we apply Theorem \ref{thm:Structure-Auto}.

\begin{proposition}\label{prop:recurrent=>not-auto}
	Let $(a_m)_{m\geq 0}$ be an $\N$-valued sequence satisfying a linear recurrence of the form
	\begin{equation}
	\label{eq:380} 
		a_{m + n} = \sum_{i=1}^{n} c_i a_{m+n-i}, \quad  m\geq 0
	\end{equation}
	with integer coefficients $c_i$. Suppose that  for some $k$ the set $E = \set{ a_m}{ m \in \NN_0}$ is $k$-automatic. Then $E$ is a finite union of the following standard sets: linear progressions $\set{ am +b}{ m \in \NN_0}$ with $a,b\in \N_0$; exponential progressions $\set{ a k^{ t m} + b}{m \in \NN_0}$ with $a,b\in \Q$ and $t\in \N$; and finite sets.
\end{proposition}
\begin{proof}
	We first claim that there exists a representation of $E$ as a finite union
	\begin{equation}\label{eq:381}
	E = \bigcup_{i=1}^{K_{\text{lin}}} L_i \cup \bigcup_{i=1}^{K_{\text{poly}}} P_i \cup \bigcup_{i = 1}^{K_{\text{exp}}} E_i \cup F,\end{equation}
	where $F$ is finite, $L_i = \set{ a_i m + b_i}{m \in \NN_0}$ are arithmetic progressions, $P_i = \set{ p_i(m) }{m \in \NN_0}$ are value sets of polynomials $p_i(x) \in \ZZ[x]$ with $\deg p_i \geq 2$, and $E_i$ have exponential growth in the sense that $\abs{E_i \cap [N]} \ll \log N$.

	In order to prove this claim, we begin by noting that any restriction of $(a_m)$ to an arithmetic progression $a^{(h,r)}_m = a_{hm + r}$ obeys some (minimal length) linear recurrence 
	$$
	a^{(h,r)}_{m + n'} = \sum_{i=1}^{n'} c^{(h,r)}_i a^{(h,r)}_{m+n'-i}, \quad m\geq 0
	$$
	with $n'=n'(h,r) \leq n$. Moreover, there exists a choice of $h$ such that each of that each $a^{(h,r)}_m$ is either identically zero or non-degenerate, in the sense that the associated characteristic polynomial $q^{(h,r)}(x) = x^{n'} - \sum_{i=1}^{n'} c^{(h,r)}_{i} x^{n'-i}$ has no pair of roots $\lambda,\mu \in \CC$ such that $\lambda/\mu$ is a root of unity
	(see, e.g., \cite[Theorem 1.2]{EverestPoortenShparlinskiWard-book} for a much stronger statement).  
	Hence, for the purpose of showing the existence of a representation of the form \eqref{eq:381}, we may assume  that $(a_m)$ is non-degenerate. Suppose also that $n$ is minimal, and let $\lambda_1,\dots,\lambda_r$ be the roots of $q(x) = x^n - \sum_{i=1}^n c_{i} x^{n-i}$ with $\abs{\lambda_1} \geq \abs{\lambda_2} \geq \dots$. Note that either $E$ is finite or $\abs{\lambda_1} \geq 1$.
	
	If $\abs{\lambda_1} > 1$, then by the result of Evertse \cite{Evertse-1984} and van der Poorten and Schlickewei \cite{PoortenSchlickewei-1991} (see \cite[Theorem 2.3]{EverestPoortenShparlinskiWard-book}), we have $a_m = \abs{\lambda_1}^{m + o(m)}$ as $m \to \infty$. Hence, $E$ has exponential growth, and we are done.
	
	Otherwise, if $\abs{\lambda_1} =1$, then for all $j$ we have $\abs{\lambda_j} = 1$ or $\lambda_j=0$. Kronecker's theorem \cite{Kronecker-1857} (or a standard Galois theory argument) shows that if $\lambda$ is an algebraic integer all of whose conjugates have absolute value $1$, then $\lambda$ is a root of unity. Using the general formula for the solution of a linear recurrence, we may write for sufficiently large $m$
	$$
		a_m = \sum_{j=1}^r \lambda_j^m p_j(m) = \sum_{j=1}^r b_j(m) p_j(m),
	$$
	where $p_j(x)$ are polynomials and $b_j(m)$ are periodic. Splitting $\NN_0$ into arithmetic progressions where $b_j(m)$ are constant, we conclude that $E$ is a finite union of value sets of polynomials. This again produces a representation of the form \eqref{eq:381}. 
	
	Such a representation is not unique. Splitting $P_i$ into a finite number of subprogressions and discarding those which are redundant, we may assume that $P_i \cap L_j = \emptyset$ for any $i,j$. Likewise, we may assume that $E_i \cap L_j = F \cap L_j =\emptyset$ for any $i,j$. Fix one such representation subject to these restrictions. The set 
	
	$$E' =\bigcup_{i=1}^{K_{\text{poly}}} P_i \cup \bigcup_{i = 1}^{K_{\text{exp}}} E_i \cup F = E \setminus \bigcup_{i=1}^{K_{\text{lin}}} L_i $$ is again $k$-automatic; it will suffice to show that $E'$ is a union of the standard sets mentioned above.
	
	We claim that $K_{\text{poly}} = 0$, i.e., the representation of $E$ uses no polynomial progressions of degree $ \geq 2$. Suppose for the sake of contradiction that $P = \set{p(m)}{m \in \NN_0}$ appears in one of the sets $P_i$, and write $p(m) = \sum_{i=0}^d c_i m^i$, where $c_i \in \ZZ$. Replacing $p(m)$ with $p(m+r)$ for a suitably chosen $r \in \NN_0$, we may assume that $c_i > 0$ for $0 \leq i \leq d$. For sufficiently large $t$, we have $p(k^t) = [u_d 0^{t - t_0} u_{d-1} 0^{t-t_0} u_{d-2} \cdots u_1 0^{t-t_0} u_{0}]_{k}$, where $t_0$ is a constant and $u_{i}$ is the base-$k$ expansion of $c_i$, padded by $0$'s so as to have $\abs{u_i}=t_0$. Since $p(k^t) \in E'$, from the pumping lemma \ref{lem:pumping} it follows that there is $l \in \NN$ such that for any $s_1,\dots,s_d\in\NN$ it holds that
	$$
		n(s_1,\dots, s_d) := [u_{d} 0^{ls_d} u_{d-1} 0^{ls_{d-1}} \cdots u_1 0^{ls_1} u_0]_{k} \in E'.
	$$
	For sufficiently large $S$ and a small absolute constant $\delta$ to be determined later, consider the set 
	$$
		Q(S) = \set{n(s_1,\dots,s_d)}{ s_i \in \NN, s_1 + \ldots + s_d = S, s_d \geq (1-\delta) S }, 
	$$
 	and put $N(S) := n(1,\dots,1,S-d+1) = \min Q(S)$ (for large $S$). Note that $N(S) = k^{l S + O(1)}$ and that $\max Q(S) = N(S) + O(N(S)^\delta)$. 
 	For a fixed $T_0$ and $T \to \infty$, we shall consider the cardinality of the set $ Q(T_0,T) = \bigcup_{T_0 \leq S \leq T} Q(S)$. By an elementary counting argument, we find 
 	\begin{equation}
 	\label{eq:438}
 	\abs{ Q(T_0,T)} \gg T^{d} \gg T^2. 
 	\end{equation}
 	
 	To obtain an upper bound, we separately estimate $\abs{ Q(S) \cap P_i}$ and $\abs{Q(T_0,T) \cap E_j}$ for each $i,j$.
 	
 	Suppose that $n, n' \in Q(S) \cap P_i$ with $n' > n$, so in particular $n = p_i(m)$ and $n' = p_i(m')$ for some $m,m' \gg N(S)^{1/\deg p_i}$. We then have the chain of inequalities:
 	$$
 		N(S)^\delta \gg n'-n = p_i(m') - p_i(m) \geq \min_{x \in [m,m']} \abs{ p_i'(x) } \gg N(S)^{\frac{ \deg p_i - 1}{\deg p_i}}, 
 	$$
	which is a contradiction for sufficiently large $S$, provided that $\delta < \frac{\deg p_i - 1}{\deg p_i}$ (which will hold if we put $\delta = \frac{1}{3}$). Thus, $\abs{ Q(S) \cap P_i} \leq 1$. 
	
	As for $Q(T_0,T) \cap E_j$, from the bounds on growth of $E_j$ we immediately have
	\begin{equation}
 	\label{eq:439}
 	\abs{ \bigcup_{T_0 \leq S \leq T} Q(T_0,T) \cap E_i} \ll \abs{E_i \cap [2N(T)]} \ll T.
 	\end{equation}

	 In total, using \eqref{eq:438} and \eqref{eq:439} we find that
 	\begin{equation}
 	\label{eq:440}
 	\abs{ Q(T_0, T) } \leq \sum_{S=T_0}^{T} \sum_{i=1}^{K_{\text{poly}}} \abs{ Q(S) \cap P_i} + \sum_{i=1}^{K_{\text{exp}}} \abs{ Q(T_0, T) \cap E_i} + O(1) \ll T,
 	\end{equation}
	 contradicting the previously obtained bound $\abs{ Q(T_0,T)} \gg T^2$. It follows that indeed $K_{\text{poly}} = 0$.
	
	Since $E'$ contains no polynomial or linear progressions, we have $\abs{E' \cap [N]} \ll \log N$. It follows from Proposition \ref{automclassthm} that $E'$ must be $k$-arid of rank $1$. Since all basic arid sets of rank $1$ are of the form described in the statement of the theorem, we are done. 
\end{proof}

\section{Proof of Theorem \ref{thm:main-dichotomy}}\label{sec:Dichotomy}

In this section, we derive Theorem \ref{thm:main-dichotomy} from Theorem \ref{thm:main-optimized}. Our argument is purely combinatorial and can be entirely phrased in terms of finite automata with no further recourse to dynamics. 

\begin{proposition}\label{jeszczejednoespresso-A} 
Let $A \subset \Sigma_k^*$ be an infinite arid set. Then there exists $v \in \Sigma_k^*$ such that $A \cap v \Sigma_k^*$ takes the form $$A \cap v \Sigma_k^*=\bigcup_{i=1}^p \{ v w^l u_i \mid l \in \NN_0\},$$ where $p\geq 1$, $v,w,u_i \in \Sigma_k^*$ and $w \neq \epsilon$. In particular, $A\cap v\Sigma_k^*$ is arid of rank $1$.

 Likewise,  there exists $\tilde{v}\in\Sigma_k^*$ such that $A \cap \Sigma_k^* \tilde{u}$ takes the form  $$A \cap \Sigma_k^* \tilde{u}=\bigcup_{i=1}^{\tilde{p}} \{ \tilde{v}_i  (\tilde{w})^l \tilde{u} \mid l \in \NN_0\},$$ where $\tilde{p}\geq 1$, $\tilde{v}_i,\tilde{w},\tilde{u} \in \Sigma_k^*$ and $\tilde{w} \neq \epsilon$.\end{proposition}

\begin{proof} Since the notion of an arid set is preserved under the reversal operation, it is sufficient to prove the former statement. 
	For $B \subset \Sigma_k^*$ and $v \in \Sigma_k^*$, put $v^{-1}B = \set{ u \in \Sigma_k^*}{ vu \in B}$. If $B$ is arid of rank $\leq r$, then so is $v^{-1}B$. 
	
\begin{claim*}
	Let $B \subset \Sigma_k^*$ be arid of rank $r$, and let $x_1,x_2,y \in \Sigma_k^*$ be such that  $a = \abs{y} = \abs{ x_2}$ and $y \neq x_2$. Then for sufficiently large $m$ (depending on $B,x_1,x_2,y$), $(x_1 y^m x_2)^{-1} B$ is arid of rank $\leq (r-1)$.
\end{claim*}
\begin{proof}
Replacing $B$ with $x_{1}^{-1}B$, we may assume that $x_1 = \epsilon$. 

Let $a=|y|=|x_2|$. In analogy with Remark \ref{kubekkawey}, note that there is a natural way to identify $\Sigma_{k^a}^*$ with a subset of $\Sigma_{k}^*$, and any arid set $B \subset \Sigma_k^*$ is a finite union of translates $B_i v_i$ with $v_i \in \Sigma
_k^*$ of arid sets $B_i \subset \Sigma_{k^a}^*$. Hence, it will suffice to show that if $B \subset \Sigma_{k^a}^*$ is arid of rank $r$, then for sufficiently large $m$, $B \cap y^m x_2 \Sigma_{k}^*$ is arid of rank $\leq (r-1)$. We may now replace $k$ with $k^a$ and assume that $\abs{y} = \abs{x_2} = 1$.

It will suffice to prove the claim for $B$ of the form 
$$B = \set{ v_0 w_1^{l_1} v_1 w_2^{l_2} \cdots w_r^{l_r} v_r }{ l_1, \dots, l_r \in \NN}$$
where $w_i \neq \epsilon$ for all $i$ (note that $l_i$ here are required to be strictly positive; any arid set of rank $r$ is a union of such sets and an arid set of rank $\leq (r-1)$). Now, if $m > \abs{v_0 w_1}$ then either $B \cap y^m x_2\Sigma_{k}^* = \emptyset$ (in which case we are trivially done) or $B \cap y^m x_2\Sigma_{k}^* \neq \emptyset$ and both $v_0$ and $w_1$ is a power of $y$. In the latter case, we further conclude that $x_2$ appears in $v_1w_2$ (else $B$ would have rank $\leq (r-1)$), which is necessarily of the form $y^{b}x_2 v_1'$ with $b \in \NN_0$. Hence 
$$(y^{m} x_2)^{-1}B =  \set{ v_1' w_2^{l_2-1} v_2 w_3^{l_3}\cdots w_r^{l_r} v_r }{ l_2, \dots, l_r \in \NN}$$
is arid of rank $\leq (r-1)$. 
\end{proof}

The proof of the proposition is now a simple induction on the rank $r$ of $A$. Since $A$ is infinite, we have $r \geq 1$. 

If $r = 1$, then $A$ takes the form $\bigcup_{i=1}^r \set{ v_i w^l_i u_i }{l \in \NN_0}$, where $w_i \neq \epsilon$ for at least one $i$, say $i = 1$. Then $A \cap v \Sigma_k^*$ takes the required form for $v = v_1 w_1^m$ for $m$ large enough.

If $r > 1$, then we may find a rank $2$ basic arid set $$B=\set{ v_0 w_1^{l_1} v_1 w_2^{l_2} v_2 }{ l_1,l_2 \in \NN_0}$$ contained in $A$. Without loss of generality, we may assume that $|w_1|=|w_2|>|v_1|$. Apply the above Claim with $x_1 = v_0$, $y = w_1^{l_1}$ and $x_2$ equal to the first $\abs{y}$ symbols of $v_1 w_2^{l_2}$, where $l_2\geq l_1\geq 2$. Note that $y\neq x_2$, because otherwise by an elementary computation one could show that the rank of $B$ is $1$. Then for $m$ large enough $A' = (x_1 y^m x_2)^{-1} A$ is arid of rank $\leq (r-1)$ and infinite. By the inductive assumption, there exists $v' \in \Sigma_k^*$ such that $A'\cap v'\Sigma_k^*$ takes the required form. It remains to take $v = x_1 y^m x_2 v'$.\end{proof}

\begin{corollary}\label{jeszczejednoespresso} Let $E$ be an infinite $k$-arid set. Then there exist integers $n\geq1$, $r\geq 0$, $p\geq 1$, and words $v_1,\ldots,v_p,w,u\in \Sigma_k^*$, $w\neq \epsilon$ such that 
$$E \cap (n\Z+r) = \bigcup_{i=1}^p \{ [v_iw^lu]_k  \mid  l\in \NN_0 \}.$$ 
\end{corollary}
\begin{proof} Follows immediately from the second part of Proposition \ref{jeszczejednoespresso-A}.\end{proof}

\begin{proposition}\label{automverysparse} If the set $\{k^l \mid l\geq 0\}$ is not generalised polynomial, then neither is any infinite $k$-arid set.
\end{proposition}
\begin{proof} Assume we know that $P = \{k^l \mid l\geq 0\}$ is not generalised polynomial. Then neither is any set of the form $P_t=\{k^{tl} \mid l\geq 0\}$ for $t\geq 1$ since $P = \bigcup_{j=0}^{t-1}k^j P_t$.

Suppose that there exists an infinite $k$-arid set which is generalised polynomial. Since the class of generalised polynomial sets contains all arithmetic progressions and is closed under finite intersections, Corollary \ref{jeszczejednoespresso} allows us to assume that
$$E = \bigcup_{i=1}^p \{[v_iw^lu]_k \mid l\geq 0\}$$
 for some $p\geq 1, v_1,\ldots,v_p,w,u\in \Sigma_k^*$, $w\neq \epsilon$. Let $s=|u|, t=|w|$ and note that $$[v_iw^lu]_k=[u]_k+k^s[w]_k \frac{k^{tl}-1}{k^t-1}+[v_i]_kk^{tl+s}.$$ 
 
\renewcommand{\p}{g}
\newcommand{\q}{h}
 Let $\p$ be a generalised polynomial such that $E=\set{n\in \N_0 }{ \p(n)=0 }$ and assume further that $\p$ is a restriction of a generalised polynomial of a real variable that has no further zeros in $\R_{> 0} \setminus \N$. (To this end, replace $\p(n)$ by $\p(n)^2+\fpa{n}^2$.)  Then an easy computation shows that the polynomial 
 $$\q(n)=\p \bra{ k^s \frac{ n-[w]_k }{ k^t-1}+[u]_k } $$ 
 has as its zero set 
 $$B=\{n\in \N \mid \q(n)=0\}=\bigcup_{i=1}^p \{b_i k^{tl}\mid l\geq 0\}$$
where $b_i = [w]_k+(k^t-1)[v_i]_k$, $i=1,\ldots,p$.

The set $C=\{n\in \N_0\mid b_1 n \in B\}$ is also generalised polynomial and it has the form 
	$$C=\bigcup_{i=1}^{p} \{c_i k^{tl}\mid l\geq 0\}$$
with $c_1=1$ and $c_i=b_ik^{tl_i}/b_1$, where $l_i\geq 0$ is the smallest integer such that $b_1$ divides $b_i k^{tl_i}$. (If there is no such integer, the corresponding term is not present.)

Let $m\geq 1$ be such that $c_i <k^{tm}$ for $i=1,\ldots,p$. Replacing the set $\{c_i k^{tl}\mid l\geq 0\}$ by the union
	$$\{c_i k^{tl}\mid l\geq 0\}=\bigcup_{j=0}^{m-1} \{c_i k^{tj} k^{mtl}\mid l\geq 0\}$$
and replacing $k$ by $k^{mt}$, we may assume that $$C=\bigcup_{i=1}^{p} \{c_i k^{l}\mid l\geq 0\}$$ with $c_1=1$ and $1\leq c_i<k^2$.

Consider the set $D=\{n\in C \mid n\equiv 1 \pmod{k^2-1}\}$. The set $D$ is generalised polynomial and an integer $c_i k^l \in C$ can be an element of $D$ only if $c_i \equiv 1 \pmod{k^2-1}$ or $c_i \equiv k  \pmod{k^2-1}$. Since $1\leq c_i \leq k^2-1$, this gives $c_i=1$ or $c_i=k$ and whether the latter possibility is realised or not, we have $D=\{k^{2l} \mid l\geq 0\}$. This is a contradiction with our remark that no set of the form $P_t=\{k^{tl} \mid l\geq 0\}$, $t\geq 1$, is generalised polynomial (note that during the proof we have replaced $k$ by its power).
\end{proof}

We are now ready to finish the proof of Theorem  \ref{thm:main-dichotomy}.
\begin{proof}[Proof of Theorem \ref{thm:main-dichotomy}] The two statements in Theorem  \ref{thm:main-dichotomy} are of course mutually exclusive. Now assume that there exists a sequence $(a_n)$ which is $k$-automatic, generalised polynomial, and not ultimately periodic. By Theorem \ref{thm:main-optimized}, it nevertheless coincides with a periodic sequence $(b_n)$ except at a set of density zero. Consider the set $C=\{n\in \N_0 \mid a_n \neq b_n\}$. This set is $k$-automatic, generalised polynomial, sparse, and infinite. By Theorem \ref{thm:main-B2}, $C$ is then arid and hence by Proposition \ref{automverysparse} the set $\{k^l \mid l\geq 0\}$ is generalised polynomial as well.
\end{proof}

\section{Concluding remarks}\label{sec:Concluding}

In this section, we gather some remarks and questions which arise naturally. The question with which we begin was already alluded to in the introduction and in \cite{ByszewskiKonieczny2016}. As previously discussed, its resolution would suffice to decide if Conjecture \ref{conjecture:main} is true.
\begin{question}
	Let $k\geq 2$ be an integer. Is the set $\set{k^i}{i \geq 0}$ generalised polynomial?
\end{question}
We find this question exceptionally pertinent because of its simple formulation.

\subsection*{Morphic words}

The class of morphic words is a natural extension of the class of automatic sequences. Let $\Omega$ be a finite set. Any morphism $\varphi$ of the monoid $\Omega^*$ extends naturally to $\Omega^{\NN_0}$. A word $w \in \Omega^{\NN_0}$ (which we identify with a function $\NN_0 \to \Omega$) is a \emph{pure morphic word} if it is a fixed point of a non-trivial morphism of $\Omega^*$. A morphic word is the image $\pi \circ w \colon \NN_0 \to \Omega'$ of a pure morphic word $w$ under a coding $\pi \colon \Omega \to \Omega'$ (i.e., any set-theoretic map, not necessarily injective). Morphic words are connected with automatic sequences via the fact that $k$-automatic sequences are precisely the morphic words coming from $k$-uniform morphisms. Here, a morphism $\varphi\colon \Omega^* \to \Omega^*$ is $k$-uniform if $\abs{\varphi(u)} = k$ for all $u \in \Omega$.

We have already encountered possibly the most famous example of a non-uniform morphic word, the Fibonacci word. Recall from the introduction that the Fibonacci word $w_{\mathrm{Fib}}$ was defined as the limit of the words $w_0 := 0$, $w_1 := 01$, and $w_{i+2} := w_{i+1} w_{i}$. Directly from this definition, it is easy to see that $w_{\mathrm{Fib}}$ is fixed by the morphism $\varphi \colon \Omega^{\NN_0} \to \Omega^{\NN_0}$ given by $\varphi(0) = 01$ and $\varphi(1) = 0$.

Recall also that $w_{\mathrm{Fib}}$ is a Sturmian word. Here, a \emph{Sturmian word} is one of the form $f(n) = \floor{ \alpha (n+1) + \rho } - \floor{ \alpha n + \rho } - \floor{\alpha}$, where $\alpha,\rho \in \RR$ and $\alpha \not \in \QQ$ (for $w_{\mathrm{Fib}}$ we may take $\alpha = \rho=2- \varphi$). Some (but not all) of these sequences give rise to morphic words; see \cite{BerstelSeebold1993} for details (cf.\ also \cite{Yasutomi97, Fagnot06, BEIR07}).

In analogy with Conjecture \ref{conjecture:main}, one could ask about a classification of all morphic words which are given by generalised polynomials. We believe that examples such as the Fibonacci word are essentially the only possible ones. 

\begin{question}
	Assume that a sequence $f \colon \NN_0 \to \Omega \subset \RR$ is both a morphic word and a generalised polynomial. Is it true that $f$ is a linear combination of a number of Sturmian morphic words and an eventually periodic sequence?
\end{question}

\subsection*{Regular sequences}
We finish by presenting a generalisation of Conjecture \ref{conjecture:main} to regular sequences. We call a function $f \colon \N_0 \to \Z$ a quasi-polynomial if there exists an integer $m\geq 1$ such that the sequences $f_j$ given by $f_j(n)=f(mn+j)$, $0\leq j \leq m-1$, are polynomials in $n$. We say that a function $f \colon \N_0 \to \Z$ is ultimately a quasi-polynomial if it coincides with a quasi-polynomial except on a finite set.

\begin{question}\label{Q:mainregular}
	Assume that a sequence $f \colon \N_0 \to \Z$ is both regular and generalised polynomial. Is it then true that $f$ is ultimately a quasi-polynomial? 
\end{question}

If $f$ takes only finitely many values, then all the polynomials inducing $f_j$ are necessarily constant, and so in this case the question coincides with Conjecture \ref{conjecture:main}. 
 
\bibliographystyle{alpha}
\bibliography{bibliography}

\end{document}